\documentclass[12pt]{article}
\usepackage{amsmath,amsthm,amssymb,amsfonts}
\usepackage{enumerate}
\usepackage{xcolor}
\usepackage{authblk}
\usepackage{appendix}
\usepackage{hyperref}
\usepackage{amsmath}
\usepackage{verbatim}
\usepackage{hyperref}
\usepackage[left=1in,right=1in,top=1in,bottom=1in]{geometry}

\usepackage{graphicx}
\usepackage{subfigure}
\usepackage{booktabs}
\usepackage{multirow}
\usepackage{xcolor}
\usepackage{kvsetkeys}

\theoremstyle{plain}
\newtheorem{theorem}{\hskip\parindent Theorem}
\newtheorem{lemma}{\hskip\parindent Lemma}
\newtheorem{corollary}{\hskip\parindent Corollary}
\newtheorem{example}{\hskip\parindent Example}

\theoremstyle{definition}
\newtheorem{remark}{\hskip\parindent Remark}

\numberwithin{equation}{section}

\begin{document}

%\title{Classification of real Painlev\'{e} I solutions and asymptotics of their large positive zeros}

\title{Generalized Freud weight, discrete Painlev\'{e} I hierarchy and full asymptotics of Hankel determinants}

\author[$~\dag$]{Wen-Gao Long \thanks{Corresponding author: longwg@hnust.edu.cn}}

\author[$\dag$]{Lin-Rong Tang}
\author[$\ddag$]{Chao Min}

\affil[$\dag$]{\small School of Mathematics and Statistics,
Hunan University of Science and Technology, Xiangtan, Hunan, 411201, China}

\affil[$\ddag$]{\small School of Mathematical Sciences, Huaqiao University, Quanzhou, Fujian 362021, China}

\date{}

\maketitle

\begin{abstract}

In this paper, we investigate the orthogonal polynomials $P_{n}(x;T_{m};\lambda)$ and the Hankel determinants $D_{n}(T_{m}; \lambda)$ associated with the generalized Freud weight
\[
w(x;T_{m};\lambda) = |x|^{2\lambda+1}\exp\biggl(-\sum_{k=1}^m t_k x^{2k}\biggr),\quad
m \in \mathbb{Z}^+,\; x, t_{k} \in \mathbb{R} ,
\]
where \(T_{m}=\{t_{1},t_{2},\cdots, t_{m}\}\), $t_{m}>0$ and \(\lambda>-1\).
By employing ladder operators and compatibility conditions, we find that all members of the discrete Painlev\'{e} I hierarchy have a unified structure and the recurrence coefficient \(\beta_n\) of $P_{n}(x;T_{m};\lambda)$ satisfies the $m$-th member of the discrete Painlev\'{e} I hierarchy. The explicit representation of the nontrivial leading coefficient \(\mathrm{p}(n; T_m; \lambda)\) of $P_{n}(x;T_{m};\lambda)$ by \(\beta_n\) has also been obtained.
Besides, we derive the second-order differential equation satisfied by $P_{n}(x;T_{m};\lambda)$, the partial derivatives of the recurrence coefficients \(\beta_n\) with respect to parameters \(t_1, t_2, \dots, t_{m-1}\) and the corresponding differential identities for $D_{n}(T_{m}; \lambda)$. Based on the discrete Painlev\'{e} I hierarchy and the above differential identities, we obtain new partial differential equations satisfied by $\ln\beta_n$ and $\ln D_{n}(T_{m}; \lambda)$.
Using the discrete Painlev\'{e} I hierarchy and the asymptotic theory of linear difference equations, we derive the full asymptotic expansions of the recurrence coefficient $\beta_n$, the nontrivial leading coefficient $\mathrm{p}(n; T_m; \lambda)$, and the Hankel determinant $D_n(T_m; \lambda)$ as $n\to\infty$, for general $T_m$ and $\lambda>-1$. Notably, while the logarithmic term $\ln n$ appears in the leading-order contributions, it is absent from the remainder terms in these expansions.
We illustrate our results under the specific decic Freud weight $w(x;t_1,t_2;\lambda) =|x|^{2\lambda+1}\exp\bigl(-x^{10}-t_2x^4-t_1x^2\bigr)$.

\end{abstract}


\section{Introduction}

In random matrix theory, the statistical properties of eigenvalues of random matrices are intimately connected to orthogonal polynomials and Hankel determinants \cite{31,1}. Consider a unitary ensemble with the probability distribution
\[
P^{(n)}(M)dM=c_{n}|\det(M)^{2\lambda+1}|e^{-\operatorname{Tr}V_{0}(M)}dM,
\]
where $V_{0}(x)=t_{m}x^{2m}+\cdots+t_{1}x^2, t_{m}>0, x\in \mathbb{R}, \lambda>-1$, $M$ is an $n \times n$ Hermitian matrix and $c$ is a normalization constant. By diagonalization, we obtain the eigenvalue distribution
\[
P^{n}(M)dM = d_n \prod_{1 \leq i < j \leq n} (\lambda_i - \lambda_j)^2 \prod_{i=1}^n |\lambda_{i}|^{2\lambda+1}e^{-V_{0}(\lambda_i)} d\lambda_1\cdots d\lambda_n,
\]
where $\{\lambda_i\}_{i=1}^n$ are the eigenvalues of $M$ and $d_{n}$ is also known as the normalization constant. Moreover, the statistical properties of the eigenvalues are intimately connected to the associated Hankel determinants, which in turn can be studied via the corresponding orthogonal polynomials.
Therefore, we need to study the orthogonal polynomials and Hankel determinants associated with the generalized Freud weight,
\begin{equation}\label{eq:Freud-weight}
w(x;T_m;\lambda)=|x|^{2\lambda+1}\exp\biggl(-\sum_{k=1}^{m}t_kx^{2k}\biggr),\quad m \ge 1,\; m\in \mathbb{Z}^+,\; x\in \mathbb{R}\setminus\{0\},
\end{equation}
where $\lambda>-1$, $T_m=(t_1,t_2,\cdots,t_m)$ and $t_m>0$. This weight function generalizes several well-known cases in the literature, including the generalized Hermite weight ($m=1$) \cite[p. 43]{3}, the generalized quartic Freud weight ($m=2$) \cite{29,24,30}, the generalized sextic Freud weight ($m=3$) \cite{7,10} and the generalized higher-order Freud weight \cite{Clarkson4}.

The Hankel determinants associated with the weight \eqref{eq:Freud-weight} are defined as
\begin{equation}
\label{eq:def-Hankel-determinant}
D_n(T_m;\lambda):=\det\bigl(\mu_{j+k}(T_m;\lambda)\bigr)_{j,k=0}^{n-1}=
\begin{vmatrix}
\mu_{0}(T_m;\lambda) & \mu_{1}(T_m;\lambda) & \cdots & \mu_{n-1}(T_m;\lambda) \\
\mu_{1}(T_m;\lambda) & \mu_{2}(T_m;\lambda) & \cdots & \mu_{n}(T_m;\lambda) \\
\vdots & \vdots & \ddots & \vdots \\
\mu_{n-1}(T_m;\lambda) & \mu_{n}(T_m;\lambda) & \cdots & \mu_{2n-2}(T_m;\lambda)
\end{vmatrix},
\end{equation}
where $\mu_i(T_m;\lambda)$ are the moments:
\[
\mu_i(T_m;\lambda):=\int_{-\infty}^{+\infty}x^iw(x;T_m;\lambda)dx,\quad i=0,1,\cdots,2n-2.
\]
It is well known that the Hankel determinants are closely related to partition functions in random matrix theory. Let $Z_n(T_m;\lambda)$ be the partition function for the unitary random matrix ensemble associated with \eqref{eq:Freud-weight}:
\begin{align}
    Z_n(T_m;\lambda):=\int_{{(-\infty,+\infty)}^n}\prod_{1 \leq i < j \leq n} (x_i - x_j)^2 \prod_{k=1}^n w(x_k;T_m;\lambda) dx_k .\label{4}
\end{align}
Then using Heine's integration formula \cite{2}, we have the fundamental relation
\[
D_n(T_m;\lambda)=\frac{1}{n!}Z_n(T_m;\lambda).
\]

The orthogonal polynomials $P_n(x;T_m;\lambda)$ associated with the weight \eqref{eq:Freud-weight} are defined by
\[
k_m k_n\int_{-\infty}^{+\infty}P_m(x;T_m;\lambda)P_n(x;T_m;\lambda)w(x;T_m;\lambda)dx = \delta_{mn},
\]
where $\delta_{mn}$ is the Kronecker delta and $k_{n}$'s are the normalization constants. With this definition, the Hankel determinant can also be expressed as the product of the normalization constants $k_{n}$'s and  the square norm of $P_{n}(x;T_{m};\lambda)$ as follows \cite{5}:
\begin{align}
   D_n(T_m;\lambda)=\prod_{i=0}^{n-1}\frac{1}{k_i^2(T_m;\lambda)}=\prod_{i=0}^{n-1}h_i(T_m;\lambda),\label{eq:relation-Dn-hn}
\end{align}
where
\begin{equation}\label{eq:def-hn}
    h_n(T_m;\lambda)=\int_{-\infty}^{\infty}P_n^2(x;T_m;\lambda)w(x;T_m;\lambda)dx.
\end{equation}
Since the weight \eqref{eq:Freud-weight} is even for $x$, we have $P_n(-x;T_m;\lambda)=(-1)^nP_n(x;T_m;\lambda)$ and the following expansion \cite{3}
\begin{equation}\label{eq:Pn-expand}
P_n(x;T_m;\lambda)=x^n+\mathrm{p}(n;T_m;\lambda)x^{n-2}+\cdots,\qquad n=0,1,2,\cdots,
\end{equation}
where $\mathrm{p}(n;T_m;\lambda)$ is the nontrivial leading coefficient, with $\mathrm{p}(0;T_m;\lambda)=\mathrm{p}(1;T_m;\lambda)=0$.

The orthogonal polynomials $P_{n}(x; T_{m};\lambda)$ also satisfy the three-term recurrence relation
\begin{equation}
xP_n(x;T_m;\lambda)=P_{n+1}(x;T_m;\lambda)+\beta_nP_{n-1}(x;T_m;\lambda),
\label{three-term-recurrence}
\end{equation}
with initial conditions $P_0(x;T_m;\lambda) = 1$, $P_{-1}(x;T_m;\lambda) = 0$. From \eqref{eq:Pn-expand} and \eqref{three-term-recurrence} we obtain
\begin{align}
  \beta_n = \mathrm{p}(n;T_m;\lambda) - \mathrm{p}(n + 1; T_m;\lambda), \quad \beta_0(T_m;\lambda) = 0\label{9}
\end{align}
and a combination of \eqref{eq:def-hn} and \eqref{three-term-recurrence} yields
\begin{align}\label{eq:relation-betan-hn}
\beta_n=\frac{h_n(T_m;\lambda)}{h_{n-1}(T_m;\lambda)}.
\end{align}

The above argument explains the fundamental relationship among random matrix theory, Hankel determinants, and orthogonal polynomials \cite{31, 4}. It should be noted that the recurrence coefficients play a central role in describing the eigenvalue distribution and exhibit rich mathematical structures, including connections to integrable systems such as the Toda lattice, the Volterra lattice, and discrete Painlev\'{e} equations.
As shown in the seminal works of Magnus \cite{6} and Fokas-Its-Kitaev \cite{17,32}, the recurrence coefficients $\beta_n$ for orthogonal polynomials with Freud weights satisfy nonlinear difference equations that are members of the discrete Painlev\'{e} I hierarchy. For example, in the sextic case ($m=3$), $\beta_n$ satisfies a fourth-order difference equation known as the string equation, which is a special case of the second member of the discrete Painlev\'{e} I hierarchy ($\mathrm{dP_I^{(2)}}$). This equation also emerges in the context of two-dimensional quantum gravity and Hermitian matrix models \cite{17, 25}. Benassi and Moro \cite{26} interpreted the oscillatory (``chaotic'') behavior of recurrence coefficients in the $M^6$ model as a dispersive shock wave in the thermodynamic limit, linking the phase transitions in matrix models to dispersive hydrodynamics.

The systematic study of orthogonal polynomials with exponential weights \(\exp\{-V(x)\}\), where \(V(x)\) is an even polynomial with positive leading coefficient, originated with G\'{e}za Freud in the 1970s \cite{Freud-1976}. For the classical Freud weight \(w(x)=\exp(-x^4)\), Freud \cite{Freud-1976} conjectured that, for monic orthogonal polynomials, the large $n$ asymptotics of the recurrence coefficient behaves like \(\beta_n \sim \frac{1}{2}\sqrt{n}\). Magnus \cite{28} proved this conjecture for the general even-degree polynomial potentials. Lubinsky, Mhaskar and Saff \cite{Lubinsky-1988} provided a more general proof of Freud's conjecture for exponential weights. Using the Riemann-Hilbert approach, Deift et al. \cite{DKMVZ-1999} established the strong asymptotics for orthogonal polynomials with respect to \(w(x)=\exp\{-V(x)\}\), where \(V(x)\) is an even polynomial. They determined the leading asymptotic behavior of the recurrence coefficient and indicated that it admits a formal asymptotic expansion, though a rigorous justification of the full expansion was not provided.

Recent studies have further explored the asymptotic properties of Freud-type orthogonal polynomials and their related Hankel determinants. Applying the ladder operator approach, some asymptotic properties for the Hankel determinant and orthogonal polynomials with the weight $w(x;t_1)=\exp(-x^6-t_1x^2)$ have been studied \cite{9}. Later on, Min and Ding \cite{10} obtained the expressions for the relevant coefficients of orthogonal polynomials under the weight
\[
w(x;t_1,t_2)=\exp(-x^6-t_2x^4-t_1x^2),\quad t_1,t_2\geq0, x\in\mathbb{R}
\tag{1.10}
\]
through ladder operators and compatibility conditions. They also obtained the asymptotic expressions of the logarithmic partial derivative of the Hankel determinant using Dyson's Coulomb fluid approach \cite{10}. 

In another contribution, Clarkson, Jordaan and Kelil \cite{24} considered the generalized quartic Freud weight
\[
w(x;1,t;\lambda)=|x|^{2\lambda+1}\exp(-x^4+tx^2),\quad \lambda>-1,\; t\in\mathbb{R}\setminus\{0\},
\]
which involves a Fisher-Hartwig singularity at the origin and corresponds to our weight \eqref{eq:Freud-weight} with \(m=2\) and $t_{2}=1$. They derived explicit expressions for the recurrence coefficients in terms of parabolic cylinder functions, established the connection to the fourth Painlev\'{e} equation and discrete Painlev\'{e} I, and proved the existence and uniqueness of positive solutions. Moreover, Clarkson and Jordaan \cite{Clarkson-Jordaan-2018} subsequently analyzed the asymptotic behavior of the generalized Freud polynomials as either the degree \(n\) or the parameter \(t\) tends to infinity, and investigated properties of their zeros, including monotonicity and interlacing. These works provide a comprehensive picture for the quartic case (\(m=2\)).
Later on, orthogonal polynomials with respect to the weight $w(x;t_1;\lambda)=|x|^{\lambda} \exp(-x^6-t_1x^2)$ have been investigated in \cite{23, 8}. For the general case $w(x)=|x|^{\lambda}e^{-nV(x)}$, Charlier and Gharakhloo \cite{Charlier-Gharakhloo} derived large-$n$ asymptotics for Hankel determinants with Freud weights, revealing connections to the Riemann zeta function and Barnes' $G$-function, using the Riemann-Hilbert method. They obtained the asymptotic behaivor of the Hankel determinants with leading terms and approximated the error term by $\mathcal{O}(\frac{\log{n}}{n})$ as $n\to\infty$.

In this paper, we extend these investigations to the generalized Freud weight \eqref{eq:Freud-weight} with arbitrary $m \geq 1$, $t_{m}>0$ and $\lambda>-1$. Our main contributions are as follows:

\begin{enumerate}
    \item [(i)] \textbf{Unified structure of the discrete Painlev\'{e} I hierarchy:} Using the ladder operator approach of Chen and Ismail \cite{11,12}, we show that all members of the discrete Painlev\'{e} I hierarchy admit a unified recursive structure. The recurrence coefficient $\beta_n(T_m,\lambda)$ satisfies the $m$-th member of the hierarchy.  Specific cases of the result are appeared in literature several times \cite{ 26, 24, 7, 10, 9, 8} and our contribution is to write it in explicit form with respect to all the parameters $t_{k}, k=1,2,\cdots,m$ and $\lambda$.

    \item [(ii)] \textbf{Differential and difference equations:} We derive the second-order differential equation satisfied by the orthogonal polynomials $P_n^{(m)}(x;T_m;\lambda)$, as well as differential-difference equations for the recurrence coefficients $\beta_{n}$ with respect to the parameters $t_k, k=1,2,\cdots,m$. These equations generalize known results for the sextic and quartic Freud weights \cite{7, 25} and connect the orthogonal polynomials $P_n^{(m)}(x;T_m;\lambda)$ to integrable systems. The differential equations satisfied by the Hankel determinants with respect to the parameters $t_{k}$ are also obtained. Based on the discrete Painlev\'{e} hierarchy and differential identities satisfied by $\beta_{n}$, we derive new partial differential equations for $\ln{\beta_{n}}$ and $\ln D_{n}(T_{m}; \lambda)$.

    \item [(iii)] \textbf{Full asymptotic expansions:} For the special case $\lambda=-\frac{1}{2}$, the full asymptotic expansion of $\beta_{n}$ is indicated by Deift et al. \cite{DKMVZ-1999}. We provide a simpler way to derive the first $m-1$ coefficients in the asymptotic expansion. For the general case $\lambda>-1$, by using with the first Painlev\'{e} hierarchy and the asymptotic theory for linear difference equations, we rigorously show the full asymptotic expansion of $\beta_{n}$, $\mathrm{p}(n;T_{m};\lambda)$ and $\ln D_{n}(T_{m};\lambda)$ as $n\to\infty$. We find that the logarithm term $\ln n$ only appears in the leading terms of their asymptotic expansions. The specific decic Freud weight $w(x;t_1,t_2;\lambda) = |x|^{2\lambda+1}\exp(-x^{10} - t_2 x^4 - t_1 x^2)$ is chosen as example to illustrate the corresponding results.

\end{enumerate}

 The primary challenge of our work lies in deriving closed-form expressions for the recurrence coefficients and the associated Hankel determinants for general \(m\), as the complexity of the ladder operator equations grows rapidly with the degree of the potential function $V(x)$. Moreover, although for the low-degree cases \(m=2,3\), the formal asymptotic expansion of the recurrence coefficients, the sub-leading coefficients, and the Hankel determinants can be derived \cite{Clarkson-Jordaan-2018}, the full asymptotic expansions of them for general \(m\) and $\lambda$ have not been systematically and rigorously established in literature. Using the Riemann-Hilbert approach, one may obtain the asymptotic behavior of the Hankel determinant with respect to the weight \eqref{eq:Freud-weight} with $o(1)$ or $\mathcal{O}(\frac{\log{n}}{n})$ error approximation as $n\to\infty$ similar as in \cite{Charlier-Gharakhloo, 22}. While, in the present work, we show that $\log{n}$ only appears in the leading terms and the Hankel determinants admit a full asymptotic expansion with respect to the power series of $n^{-\frac{1}{m}}$ as $n\to\infty$.
 
The remainder of this paper is organized as follows: In Section 2, we apply the ladder operator method to derive the discrete Painlev\'{e} hierarchy satisfied by $\beta_{n}$ and the expression for the sub-leading coefficient $\mathrm{p}(n; T_{m};\lambda)$. Section 3 presents the differential equations for the orthogonal polynomials and the differential identities for the Hankel determinants $D_{n}(T_{m};\lambda)$. Section 4 is devoted to a simpler way to calculate the coefficients of the asymptotic expansion in the special case $\lambda=-\frac{1}{2}$ and the leading asymptotic analysis for $\beta_{n}, \mathrm{p}(n; T_{m};\lambda)$ and $D_{n}(T_{m};\lambda)$ in the general case $\lambda>-1$. In the last section, using the leading asymptotic behaviors obtained in Section 4, in conjunction with the discrete Painlev\'{e} I hierarchy and the asymptotic theory of the linear difference equation, we rigorously derive the full asymptotic expansions of $\beta_{n}, \mathrm{p}(n; T_{m};\lambda)$ and $D_{n}(T_{m};\lambda)$.

\section{Ladder operators and compatibility conditions}\label{sec:dPI}
           The ladder operator approach developed by Chen and Ismail \cite{11,12} is a highly useful and potential
     tool for analyzing the recurrence coefficients of orthogonal polynomials and the associated Hankel
     determinants; see, e.g. \cite{27,13,14,15,21}. Based the main ideas in \cite{11,12}, we have the following lemma. 
     \begin{lemma}\label{lem-ladder-operator}
    The orthogonal polynomials associated with the
     weight \eqref{eq:Freud-weight} satisfy the lowering and raising operators
     \begin{equation}
     	\left(\frac{d}{dx}+B_n(x)\right)P_n(x; T_{m};\lambda)=\beta_nA_n(x)P_{n-1}(x; T_{m};\lambda),\label{11}
     \end{equation}
     \begin{equation}
     	\left(\frac{d}{dx}-B_n(x)-V'(x)\right)P_{n-1}(x; T_{m};\lambda)=-A_{n-1}(x)P_{n}(x; T_{m};\lambda),\label{12}
     \end{equation}
     where 
        \begin{equation}
     	V(x)=-\ln w(x; T_{m};\lambda)=V_{0}(x)-(2\lambda+1)\ln \lvert x \rvert=\sum_{k=1}^{m}t_kx^{2k}-(2\lambda+1)\ln \lvert x \rvert,\label{18}
     \end{equation}
     and
     \begin{equation}
     	A_n(x):=\frac{1}{h_n}\int_{-\infty}^{+\infty}\frac{V_{0}'(x)-V_{0}'(y)}{x-y}P_n^2(y; T_{m};\lambda)w(y; T_{m};\lambda)dy,\label{13}
     \end{equation}
     \begin{equation}
     \begin{split}
     		B_n(x):=&\frac{1}{h_{n-1}}\int_{-\infty}^{+\infty}\frac{V_{0}'(x)-V_{0}'(y)}{x-y}P_n(y; T_{m};\lambda)P_{n-1}(y; T_{m};\lambda)w(y; T_{m};\lambda)dy+\frac{2\lambda+1}{x}r_{n,0},
            \label{14}
            \end{split}
     \end{equation}
with 
     \begin{align}
     	r_{n,0}=\frac{1}{h_{n-1}}\int_{-\infty}^{+\infty}y^{-1}P_n(y; T_{m};\lambda)P_{n-1}(y; T_{m};\lambda)w(y; T_{m};\lambda)dy=\frac{1-(-1)^n}{2}.\label{24.01}
     \end{align}
     \end{lemma} 

\begin{proof}
For general $t_{1},\dots,t_{m}\in\mathbb{R},t_{m}>0$ and $\lambda>-1$, let $P_{n}(x):=P_{n}(x;T_{m};\lambda)$ be the monic orthogonal polynomial of degree $n$ with respect to the generalized Freud weight \eqref{eq:Freud-weight}. Then we have the expansion
\begin{equation}\label{Pn'-by-c-nk}
P_{n}'(x)=\sum_{k=0}^{n-1} c_{n,k}P_{k}(x),
\end{equation}
where the coefficients satisfy $c_{n,k}=0$ whenever $n+k$ is even, and
\begin{equation}\label{eq:c-nk}
\begin{aligned}
c_{n,k}
&=\frac{1}{h_{k}}\int_{-\infty}^{+\infty}P_{n}'(x)P_{k}(x)w(x;T_{m};\lambda)\,dx \\
&=-\int_{-\infty}^{+\infty}P_{n}(x)\bigl(P_{k}(x)w(x;T_{m};\lambda)\bigr)'dx \\
&=\int_{-\infty}^{+\infty}P_{n}(x)P_{k}(x)
\left(V_{0}'(x)-\frac{2\lambda+1}{x}\right)w(x;T_{m};\lambda)\,dx,
\end{aligned}
\end{equation}
when $n+k$ is odd.

For convenience, we aim to use the last expression in \eqref{eq:c-nk} for $c_{n,k}$ uniformly for all $k=1,\dots,n-1$. This is valid when both $n$ and $k$ are odd, since in that case $V_{0}'(x)-\frac{2\lambda+1}{x}$ is an odd function and $w$ is even. When both $n$ and $k$ are even, the same formula remains correct provided we impose the convention
\begin{equation}\label{eq:Cauchy-integral-w/x}
\int_{-\infty}^{+\infty}\frac{(2\lambda+1)w(x;T_{m};\lambda)}{x}\,dx=0,
\end{equation}
even though the integral is not integrable at $x=0$ for $-1<\lambda<-\frac12$. This requires that the integral \eqref{eq:Cauchy-integral-w/x} be understood in the sense of the Cauchy principal value.

Substituting the expression for $c_{n,k}$ into \eqref{Pn'-by-c-nk} and applying the Christoffel--Darboux formula, we obtain
\begin{align}
P_{n}'(x)
&=\int_{-\infty}^{+\infty}P_{n}(y)
\left(\sum_{k=0}^{n-1}\frac{P_{k}(y)P_{k}(x)}{h_{k}}\right)
\left(V_{0}'(y)-\frac{2\lambda+1}{y}\right)w(y;T_{m};\lambda)\,dy \notag \\
&=\int_{-\infty}^{+\infty}P_{n}(y)
\left(\frac{P_{n}(x)P_{n-1}(y)-P_{n}(y)P_{n-1}(x)}{h_{n-1}(x-y)}\right)
\left(V_{0}'(y)-V_{0}'(x)\right)w(y;T_{m};\lambda)\,dy \notag \\
&\quad -\int_{-\infty}^{+\infty}P_{n}(y)
\left(\frac{P_{n}(x)P_{n-1}(y)-P_{n}(y)P_{n-1}(x)}{h_{n-1}}\right)
\left(\frac{2\lambda+1}{xy}\right)w(y;T_{m};\lambda)\,dy \notag \\
&=\beta_n A_n(x)P_{n-1}(x)
-\left(B_n(x)-\frac{2\lambda+1}{x}r_{n,0}\right)P_{n}(x) \notag \\
&\quad -\frac{2\lambda+1}{x}
\left(\int_{-\infty}^{+\infty}
\frac{P_{n}(x)P_{n-1}(y)P_n(y)-P_n^2(y)P_{n-1}(x)}{y h_{n-1}}
w(y;T_{m};\lambda)\,dy\right) \notag \\
&=\beta_n A_n(x)P_{n-1}(x)-B_n(x)P_n(x). \label{ladder operator one}
\end{align}
Furthermore, we have
\begin{align}
B_n(x)+B_{n+1}(x)
&=\frac{1}{h_{n}}\int_{-\infty}^{+\infty}
\frac{V_0'(x)-V_0'(y)}{x-y}\, y\, P_n^2(y) w(y;T_m;\lambda)\,dy
+\frac{2\lambda+1}{x} \notag \\
&=\frac{1}{h_{n}}\int_{-\infty}^{+\infty}
\bigl(V_0'(y)-V_0'(x)\bigr)P_n^2(y)w(y;T_m;\lambda)\,dy \notag \\
&\quad +\frac{x}{h_n}\int_{-\infty}^{+\infty}
\frac{V_0'(x)-V_0'(y)}{x-y}P_n^2(y)w(y;T_m;\lambda)\,dy
+\frac{2\lambda+1}{x} \notag \\
&=-V_0'(x)+xA_n(x)+\frac{2\lambda+1}{x} \notag \\
&=xA_n(x)-V'(x). \label{compatibility condition one}
\end{align}

Now, applying the ladder operator formula \eqref{ladder operator one} to $P_{n-1}$, we get
\begin{align*}
P_{n-1}'(x)=\beta_{n-1}A_{n-1}(x)P_{n-2}(x)-B_{n-1}(x)P_{n-1}(x).
\end{align*}
Using the three-term recurrence relation and replacing $\beta_{n-1}P_{n-2}(x)$ by $xP_{n-1}(x)-P_n(x)$, we obtain
\begin{align*}
P_{n-1}'(x)=\bigl(xA_{n-1}(x)-B_{n-1}(x)\bigr)P_{n-1}(x)-A_{n-1}(x)P_n(x).
\end{align*}
Finally, making use of the compatibility condition \eqref{compatibility condition one}, we arrive at
\begin{align}
P_{n-1}'(x)=\bigl(B_n(x)+V'(x)\bigr)P_{n-1}(x)-A_{n-1}(x)P_n(x). \label{ladder operator two}
\end{align}
This completes the proof of Lemma \ref{lem-ladder-operator}.
\end{proof}

%     Note that we have suppressed the $T_m,\lambda$ dependence of many quantities such as $\beta_n$,$w(x)$,$A_n(x)$,
%$B_n(x)$,$P_n(x)$ and $h_n$ for
%     convenience.

Combining the definitions of the functions $A_n(x)$ and $B_n(x)$ above and the three-term recurrence relation \eqref{three-term-recurrence}, it can be shown that $A_n(x)$ and $B_n(x)$ satisfy the following compatibility conditions
     \begin{equation}
     	B_{n+1}(x)+B_{n}(x)=xA_n(x)-V'(x),\label{15}
     \end{equation}
     \begin{equation}
     	1+x(B_{n+1}(x)-B_{n}(x))=\beta_{n+1}A_{n+1}(x)-\beta_{n}A_{n-1}(x).\label{16}
     \end{equation}
It follows by combining the above two equations that
     \begin{equation}
     	B_n^2(x)+V'(x)B_n(x)+\sum_{i=0}^{n-1}A_i(x)=\beta_nA_n(x)A_{n-1}(x).\label{17}
     \end{equation}

Making use of the explicit definition of $V_{0}(x)$ in \eqref{18}, we have
     \begin{equation}
     	\frac{V_{0}^{'}(x)-V_{0}^{'}(y)}{x-y}=\sum_{k=1}^{m}\sum_{i=0}^{2k-2}2kt_kx^{2k-2-i}y^{i}.\label{19}
     \end{equation}
Substituting \eqref{19} into the definitions of $A_n(x)$ and $B_n(x)$ in \eqref{13} and \eqref{14}, we obtain
     \begin{equation}
     	A_n(x)=\sum_{k=1}^{m}\sum_{i=0}^{k-1}2kt_kR_{n,i}x^{2(k-1-i)},\label{20}
     \end{equation}
     \begin{equation}
     	B_n(x)=\sum_{k=2}^{m}\sum_{i=1}^{k-1}2kt_kr_{n,i}x^{2(k-i-1)+1}+\frac{2\lambda+1}{x}r_{n,0},\label{21}
     \end{equation}
     where
     \begin{equation}
     	R_{n,i}:=\frac{1}{h_n}\int_{-\infty}^{+\infty}y^{2i}P_n^2(y; T_{m};\lambda)w(y; T_{m};\lambda)dy,\label{22}
     \end{equation}
     \begin{equation}
     	r_{n,i}:=\frac{1}{h_{n-1}}\int_{-\infty}^{+\infty}y^{2i-1}P_n(y; T_{m};\lambda)P_{n-1}(y; T_{m};\lambda)w(y; T_{m};\lambda)dy.\label{23}
     \end{equation}
Combining \eqref{20}, \eqref{21} and \eqref{15}, we have
     \begin{equation}
     	R_{n,i}=r_{n,i}+r_{n+1,i}.\label{24}
     \end{equation}
     where $i=1,2,\cdots,m-1$. Here, it should be noted from the definition of $h_{n}$ in \eqref{eq:def-hn} that when $i = 0$,
     \begin{align}
     	R_{n,0}=\frac{1}{h_n}\int_{-\infty}^{+\infty}P_n^2(y; T_{m};\lambda)w(y; T_{m};\lambda)dy=1.\notag
     \end{align}
When $i=1$, making use of the three term recurrence relationship \eqref{three-term-recurrence}, we obtain
     \begin{align}
     	r_{n,1}=&\frac{1}{h_{n-1}}\int_{-\infty}^{+\infty}(P_{n+1}(y; T_{m};\lambda)+\beta_{n}P_{n-1}(y; T_{m};\lambda))P_{n-1}(y; T_{m};\lambda)w(y; T_{m};\lambda)dy\notag\\
     	=&\beta_{n}.\label{24.09}
     \end{align}
     Substituting \eqref{20} and \eqref{21} into \eqref{17},we have
     \begin{equation}
	\begin{split}
&\beta_n\cdot\left(\sum_{k=1}^{m}\sum_{i=0}^{k-1}2kt_kR_{n,i}x^{2(k-1-i)}\right)\cdot\left(\sum_{k=1}^{m}\sum_{i=0}^{k-1}2kt_kR_{n-1,i}x^{2(k-1-i)}\right)\\
     	&-\left(\sum_{k=2}^{m}\sum_{i=1}^{k-1}2kt_kr_{n,i} x^{2(k-i-1)+1}+\frac{(2\lambda+1)r_{n,0}}{x}\right)^2\\
     	&-\left(\sum_{k=1}^{m}2kt_kx^{2k-1}+\frac{2\lambda+1}{x}\right)\left(\sum_{k=2}^{m}\sum_{i=1}^{k-1}2kt_kr_{n,i}x^{2(k-i-1)+1}
     	+\frac{(2\lambda+1)r_{n,0}}{x}\right)\\
     	&=\sum_{j=0}^{n-1}\sum_{k=1}^{m}\sum_{i=0}^{k-1}2kt_kR_{j,i}x^{2(k-1-i)}.\label{24.1}
        \end{split}
     \end{equation}

     It  can be found that the highest degree of $x$ on both sides of the above equation is $4m-4$. By comparing the coefficients of $x^{2(2m-i-1)}, i=1,2,\cdots,m$ on both sides of the equation \eqref{24.1}, we can determine the following $m$ equations:
     \begin{align}
     	&x^{2(2m-2)}: &&\beta_{n}-r_{n,1}=0.\label{24.12}\tag{$b_1$}\\
     	&x^{2(2m-3)}: &&\beta_{n}(R_{n,1}+R_{n-1,1})-r_{n,2}-(r_{n,1})^2=0.\label{24.13}\tag{$b_2$}\\
     	&x^{2(2m-4)}: &&\beta_{n}\!\!\!\sum_{\substack{i+j=2\\0\le i,j \le 2}}R_{n,i}R_{n-1,j}-r_{n,3}-\!\!\!\sum_{\substack{i+j=3\\1\le i,j \le 2}}r_{n,i}r_{n,j}=0.\label{24.14}\tag{$b_3$}\\
     	&&\dots\notag\\
     		&x^{2m}: &&\beta_{n}\!\!\!\sum_{\substack{i+j=m-2\\0\le i,j \le m-2}}R_{n,i}R_{n-1,j}-r_{n,m-1}-\!\!\!\sum_{\substack{i+j=m-1\\1\le i,j \le m-2}}r_{n,i}r_{n,j}=0.\label{24.15}\tag{$b_{m-1}$}\\
     	&x^{2(m-1)}: &&\sum_{k=1}^{m}2kt_k\left(\beta_{n}\!\!\!\!\sum_{\substack{i+j=k-1\\0\le i,j \le k-1}}R_{n,i}R_{n-1,j}-\!\!\!\!\!\!\sum_{\substack{i+j=k\\1\le i,j \le k-1}}r_{n,i}r_{n,j}\right)-(2\lambda+1)r_{n,0}=n.\label{24.16}\tag{$b_{m}$}
     \end{align}
We also find that $R_{n,m-1}$ and $r_{n,m-1}$ can be solved from equations \eqref{24.12} to \eqref{24.15}, and they can all be expressed in terms of $\beta_{n}$. Precisely, we have:
     \begin{equation}
     \begin{cases} & r_{n,1}= \beta_{n}, \notag\\
     	 & R_{n,k}=r_{n,k}+r_{n+1,k},\notag\\
     	& r_{n,k}= \beta_{n} \sum\limits_{\substack{i+j=k-1\\0\le i,j \le k-1}}R_{n,i}R_{n-1,j}-\sum\limits_{\substack{i+j=k\\1\le i,j \le k-1}}r_{n,i}r_{n,j},
        \end{cases}
        \quad k=2,\cdots,m-1.
        \label{24.1712}
     \end{equation}

Before comparing the coefficients of $x^{2(m-s-1)}, s=1,2,\cdots m-1$ in both sides of \eqref{24.1}, we proceed to define
     \begin{equation}
     	r_{n,m} = \beta_{n}\!\!\!\sum_{\substack{i+j=m-1\\0\le i,j \le m-1}}R_{n,i}R_{n-1,j}-\!\!\!\sum_{\substack{i+j=m\\1\le i,j \le m-1}}r_{n,i}r_{n,j}.\label{24.1712001}
     \end{equation}
     \begin{equation}
     	r_{n,m+s} = \beta_{n}\!\!\!\sum_{\substack{i+j=m+s-1\\0\le i,j \le m-1}}R_{n,i}R_{n-1,j}-\!\!\!\sum_{\substack{i+j=m+s\\1\le i,j \le m-1}}r_{n,i}r_{n,j}.\label{24.1712002}
     \end{equation}
     where $s=1,2,\cdots,m-1$. Then we have

     \begin{align}
     	x^{2(m-2)}:\quad &\sum_{k=1}^{m}2kt_kr_{n,k+1}-\sum_{k=1}^{m-1}2kt_k\beta_n(r_{n-1,k}+r_{n+1,k})+(2\lambda+1)(1-r_{n,0})r_{n,1}=\sum_{j=0}^{n-1}R_{j,1}. \label{24.17}\tag{$b_{m+1}$}
     \end{align}
     \begin{align}
     		x^{2(m-3)}:\quad&\sum_{k=1}^{m}2kt_kr_{n,k+2}-\sum_{k=1}^{m-1}2kt_k\beta_n(R_{n,1}r_{n-1,k}+\beta_nr_{n,k}+R_{n-1,1}r_{n+1,k})\notag\\
     		&-\sum_{k=1}^{m-2}2kt_k\beta_n(r_{n-1,k+1}+r_{n+1,k+1})+(2\lambda+1)(1-r_{n,0})r_{n,2}=\sum_{j=0}^{n-1}R_{j,2}. \label{24.1703}\tag{$b_{m+2}$}
     \end{align}
     \begin{align}
     	x^{2(m-s-1)}:\quad &\sum_{k=1}^{m-1}\!\!\sum_{\substack{i=0\\k+i\le m-1}}^{s-1}\left(4kt_kr_{n,k+i}r_{n,s-i}-2\beta_{n}kt_k(R_{n,k+i}R_{n-1,s-i-1}+R_{n,s-i-1}R_{n-1,k+i})\right)\notag\\
        &+\sum_{k=1}^{m}2kt_kr_{n,k+s}+(2\lambda+1)(1-r_{n,0})r_{n,s}=\sum_{j=0}^{n-1}R_{j,s}, \qquad 1\leq s\leq m-1.\label{24.1704}\tag{$b_{m+s}$}
     \end{align}

Basing on the above anlaysis, we have the following theorem.

\begin{theorem}\label{thm-dPI}
For any $m\in\mathbb{N}^{+}, t_{m}>0$ and $\lambda>-1$, the recurrence coefficient $\beta_{n}$ defined in \eqref{three-term-recurrence} satisfies the following $2m-2$ order nonlinear difference equation
     \begin{equation}
     	\sum_{k=1}^{m}2kt_kF_{k,n}(\beta_{n-k+1},\beta_{n-k+2},\cdots,\beta_{n+k-2},\beta_{n+k-1})-\left(\lambda+\frac{1}{2}\right)(1-(-1)^n)=n, \label{eq:definition-dPI}
     \end{equation}
      which is the the $(m-1)$-th member of the discrete Painlev\'{e} I hierarchy, where
      $F_{k,n}(\cdot)$ is given recursively by:
      \begin{align}
      	F_{1,n} &= \beta_{n}, \notag\\
      	F_{k,n} &= F_{1,n}(F_{k-1,n-1}+2F_{k-1,n}+F_{k-1,n+1})\notag\\
      	&+F_{1,n}\sum_{\substack{i+j=k-1\\1\le i,j \le k-2}}(F_{i,n}+F_{i,n+1})(F_{j,n-1}+F_{j,n})
      	-\sum_{\substack{i+j=k\\1\le i,j \le k-1}}F_{i,n}F_{j,n}.\label{eq:F-recurrence}
      \end{align}
      where $k=2,\cdots,m$.  Moreover, we find that for each $k$, $F_{k,n}$ is a polynomial of $\beta_{n-k+1}, \cdots, \beta_{n},$ $ \cdots, \beta_{n+k-1}$ with all the coefficients being non-negative integers.
\end{theorem}

\begin{proof}
According to \eqref{24.1712}, we have
       \begin{align}
       	r_{n,k}=&\beta_{n}\!\!\!\sum_{\substack{i+j=k-1\\0\le i,j \le k-1}}R_{n,i}R_{n-1,j}-\!\!\!\sum_{\substack{i+j=k\\1\le i,j \le k-1}}r_{n,i}r_{n,j}\notag\\
       	=&\beta_{n}\left(R_{n,k-1}+R_{n-1,k-1}+\sum_{\substack{i+j=k-1\\1\le i,j \le k-2}}R_{n,i}R_{n-1,j}\right)-\!\!\!\sum_{\substack{i+j=k\\1\le i,j \le k-1}}r_{n,i}r_{n,j}\notag\\
       	=&\beta_{n}\left(r_{n-1,k-1}+2r_{n,k-1}+r_{n+1,k-1}+\sum_{\substack{i+j=k-1\\1\le i,j \le k-2}}(r_{n,i}+r_{n+1,i})(r_{n-1,j}+r_{n,j})\right)\notag\\
       	&-\!\!\!\sum_{\substack{i+j=k\\1\le i,j \le k-1}}r_{n,i}r_{n,j}\label{24.1713002}
       \end{align}
Hence, we get \eqref{eq:F-recurrence} by denoting
       \begin{equation}
       	F_{k,n}=r_{n,k},\qquad k=1,2,\cdots,m. \label{24.1713001}
       \end{equation}

Next, we show that $F_{k,n}$ depends only on $2k-1$ variables $\beta_{n-k+1},\beta_{n-k+2},\cdots,\beta_{n+k-2},\beta_{n+k-1}.$ We proceed by mathematical induction on $k$.

When $k=1$, $F_{1,n}=\beta_{n}$, it depends only on $\beta_{n}$. Assume that $F_{k-1,n}$ depends only on $2k-3$ variables $\beta_{n-k+2},\beta_{n-k+3},\cdots,\beta_{n+k-3},\beta_{n+k-2}$, then we have $F_{k-1,n+1}$ depends only on $2k-3$ variables $\beta_{n-k+3},\beta_{n-k+4},\cdots,\beta_{n+k-2},\beta_{n+k-1}$ and $F_{k-1,n-1}$ depends only on $2k-3$ variables $\beta_{n-k+1},\beta_{n-k+2},\cdots,\beta_{n+k-4},\beta_{n+k-3}$. Combining these facts and \eqref{eq:F-recurrence}, we have
\begin{align}
       F_{k,n}=F_{k,n}(\beta_{n-k+1},\beta_{n-k+2},\cdots,\beta_{n},\cdots,\beta_{n+k-2},\beta_{n+k-1})\label{24.1713003}
\end{align}
       Substituting \eqref{24.01} \eqref{eq:F-recurrence}, \eqref{24.1713002} and \eqref{24.1713003} into \eqref{24.16}, we can get \eqref{eq:definition-dPI}.

Finally, we prove that, for each $k\in\mathbb{N}^{+}$, $F_{k,n}$ is a polynomial of $\beta_{n-k+1}, \cdots, \beta_{n},$ $ \cdots, \beta_{n+k-1}$ with all the coefficients being non-negative integer. Consider the general three-term recurrence relation
\begin{align*}
 xP_n(x;T_m;\lambda)=P_{n+1}(x;T_m;\lambda)+\beta_nP_{n-1}(x;T_m;\lambda),
\end{align*}
Multiplying $x^2$ on both sides and making use of the above recurrence relation again, we have
\begin{align*}
 x^3 P_n(x;T_m;\lambda)=&P_{n+3}(x;T_m;\lambda)+(\beta_{n+2}+\beta_{n+1}+\beta_{n})P_{n+1}(x;T_m;\lambda)\\
 &+\beta_{n}(\beta_{n-1}+\beta_{n}+\beta_{n+1})P_{n-1}(x;T_m;\lambda)+\beta_{n}\beta_{n-1}\beta_{n-2}P_{n-3}(x;T_m;\lambda),
 \end{align*}
 Continuing the above operations, we have
 \begin{align}\label{eq:x2m-1P}
     x^{2k_{1}+1}P_n(x;T_m;\lambda)&=P_{n+2k_{1}+1}(x;T_m;\lambda)+Q_{k_{1}}^{(k_{1}-1)}(n)P_{n+2k_{1}-1}(x;T_m;\lambda)+\cdots\\
     &+Q_{k_{1}}^{(-1)}(n)P_{n-1}(x;T_m;\lambda)+\cdots+Q_{k_{1}}^{(-k_{1}-1)}(n)P_{n-2k_{1}-1}(x;T_m;\lambda).
 \end{align}
 where $k_{1}=0, 1,\cdots, m$ and $Q_{k_{1}}^{(k_2)}(n), k_{2}=-k_{1}-1,\cdots, k_{1}-1$ are polynomials of $$\beta_{n-m+1},\cdots,\beta_{n+m-1}$$ with all coefficients being non-negative integers.

 Combining \eqref{eq:x2m-1P} with \eqref{23}, we have
 \begin{equation*}
    F_{k,n}=Q_{2k-2}^{(-1)}(n), \qquad k=1,2,3,\cdots,m.
 \end{equation*}
Therefore, $F_{k,n}$ is a polynomial of $\beta_{n-k+1}, \cdots, \beta_{n},$ $ \cdots, \beta_{n+k-1}$ with all the coefficients being non-negative integers.
\end{proof}

\begin{remark}\label{rem-explnation-thm-1}
\textit{The relation between the recurrence coefficient $\beta_{n}$ and the discrete Painlev\'{e} equations was first observed in \cite{17} by considering the special weight $w(x)=\exp(-t_{3}x^6-t_{2}x^4-t_{1}x^{2})$. Here, our equation \eqref{eq:definition-dPI} can be regarded as the general form of all members of the discrete Painlev\'{e} I hierarchy \cite{16}. It should be mentioned that particular cases of \eqref{eq:definition-dPI} have been appeared in several literature \cite{26, 24, 7, 25, 6, 10, 9, 8}. Here, we consider (for general $m\in\mathbb{N}^{+}$) all the parameters $t_{j}, j=1,2,\cdots,m$ and $\lambda$, and derive the genearal discrete Painlev\'{e} I hierarchy in a unified form using the ladder operator approach.}
       	
When $\lambda=-\frac{1}{2}$, equation \eqref{eq:definition-dPI} is appeared in \cite{26}. When $t_{2}=t_{3}=\cdots=t_{m-1}=0$, it has been obtained in \cite{24, 25}.

When $m=1$, equation \eqref{eq:definition-dPI} becomes
       		\begin{align*}
       			2t_1\beta_{n}-\left(\lambda+\frac{1}{2}\right)(1-(-1)^n)=n,
       		\end{align*}
       		which is the trivial linear nonautonomous equation, $d_0P_I.$
       		
 When $m=2$, equation \eqref{eq:definition-dPI} becomes
       		\begin{align*}
       			4t_2\beta_{n}(\beta_{n-1}+\beta_{n}+\beta_{n+1})+2t_1\beta_{n}-\left(\lambda+\frac{1}{2}\right)(1-(-1)^n)
       			=n,
       		\end{align*}
       		which is the second-order equation, $d_1P_I.$
       		
When $m=3$, equation \eqref{eq:definition-dPI} becomes
       		\begin{align*}
       			&6t_3\beta_{n}(\beta_{n-2}\beta_{n-1}+\beta_{n-1}^2+2\beta_{n-1}\beta_{n}+\beta_{n-1}\beta_{n+1}+\beta_{n}^2+2\beta_{n}\beta_{n+1}+\beta_{n+1}^2+\beta_{n+1}\beta_{n+2})\\
       			&+4t_2\beta_{n}(\beta_{n-1}+\beta_{n}+\beta_{n+1})+2t_1\beta_{n}-\left(\lambda+\frac{1}{2}\right)(1-(-1)^n)=n,
       		\end{align*}
       		which is the second member of the discrete Painlev\'{e} I hierarchy; see \cite[(2.15)]{10}.

            When $m=4$, equation \eqref{eq:definition-dPI} becomes
            \begin{align*}
              &8t_4(\beta_{n-2}^2\beta_{n-1}\beta_{n}
+2\beta_{n-2}\beta_{n-1}^2\beta_{n}
+2\beta_{n-2}\beta_{n-1}\beta_{n}^2
+\beta_{n-2}\beta_{n-1}\beta_{n}\beta_{n+1}
+\beta_{n-3}\beta_{n-2}\beta_{n-1}\beta_{n}
\\&
+\beta_{n-1}^3\beta_{n}
+3\beta_{n-1}^2\beta_{n}^2
+\beta_{n-1}^2\beta_{n}\beta_{n+1}
+3\beta_{n-1}\beta_{n}^3
+4\beta_{n-1}\beta_{n}^2\beta_{n+1}
+3\beta_{n}^2\beta_{n+1}^2
+\beta_{n-1}\beta_{n}\beta_{n+1}^2
\\&
+\beta_{n-1}\beta_{n}\beta_{n+1}\beta_{n+2}
+3\beta_{n}^3\beta_{n+1}
+\beta_{n}^4
+2\beta_{n}^2\beta_{n+1}\beta_{n+2}
+\beta_{n}\beta_{n+1}^3
+2\beta_{n}\beta_{n+1}^2\beta_{n+2}
+\beta_{n}\beta_{n+1}\beta_{n+2}^2
\\&
+\beta_{n}\beta_{n+1}\beta_{n+2}\beta_{n+3})+6t_3\beta_{n}(\beta_{n-2}\beta_{n-1}+\beta_{n-1}^2+2\beta_{n-1}\beta_{n}+\beta_{n-1}\beta_{n+1}+\beta_{n}^2
+2\beta_{n}\beta_{n+1}+\beta_{n+1}^2
\\&
+\beta_{n+1}\beta_{n+2})+4t_2\beta_{n}(\beta_{n-1}+\beta_{n}+\beta_{n+1})+2t_1\beta_{n}-\left(\lambda+\frac{1}{2}\right)(1-(-1)^n)=n.
            \end{align*}

            When $m=5$, equation \eqref{eq:definition-dPI} becomes
            \begin{align*}
                &10t_5(\beta_{n-3}^2\beta_{n-2}\beta_{n-1}\beta_{n}
+2\beta_{n-3}\beta_{n-2}^2\beta_{n-1}\beta_{n}
+2\beta_{n-3}\beta_{n-2}\beta_{n-1}^2\beta_{n}
+2\beta_{n-3}\beta_{n-2}\beta_{n-1}\beta_{n}^2
\\&
+\beta_{n-3}\beta_{n-2}\beta_{n-1}\beta_{n}\beta_{n+1}
+\beta_{n-4}\beta_{n-3}\beta_{n-2}\beta_{n-1}\beta_{n}
+\beta_{n-2}^3\beta_{n-1}\beta_{n}
+3\beta_{n-2}^2\beta_{n-1}^2\beta_{n}+2\beta_{n-2}^2\beta_{n-1}\beta_{n}^2
\\&
+\beta_{n-2}^2\beta_{n-1}\beta_{n}\beta_{n+1}
+3\beta_{n-2}\beta_{n-1}^3\beta_{n}
+6\beta_{n-2}\beta_{n-1}^2\beta_{n}^2
+2\beta_{n-2}\beta_{n-1}^2\beta_{n}\beta_{n+1}+3\beta_{n-2}\beta_{n-1}\beta_{n}^3
\\&
+4\beta_{n-2}\beta_{n-1}\beta_{n}^2\beta_{n+1}
+\beta_{n-2}\beta_{n-1}\beta_{n}\beta_{n+1}^2
+\beta_{n-2}\beta_{n-1}\beta_{n}\beta_{n+1}\beta_{n+2}
+\beta_{n-1}^4\beta_{n}+4\beta_{n-1}^3\beta_{n}^2
\\&
+\beta_{n-1}^3\beta_{n}\beta_{n+1}
+6\beta_{n-1}^2\beta_{n}^3
+6\beta_{n-1}^2\beta_{n}^2\beta_{n+1}
+\beta_{n-1}^2\beta_{n}\beta_{n+1}^2
+\beta_{n-1}^2\beta_{n}\beta_{n+1}\beta_{n+2}
+4\beta_{n-1}\beta_{n}^4
\\&
+9\beta_{n-1}\beta_{n}^3\beta_{n+1}
+6\beta_{n-1}\beta_{n}^2\beta_{n+1}^2
+4\beta_{n-1}\beta_{n}^2\beta_{n+1}\beta_{n+2}
+\beta_{n-1}\beta_{n}\beta_{n+1}^3
+2\beta_{n-1}\beta_{n}\beta_{n+1}^2\beta_{n+2}
\\&
+\beta_{n-1}\beta_{n}\beta_{n+1}\beta_{n+2}^2
+\beta_{n-1}\beta_{n}\beta_{n+1}\beta_{n+2}\beta_{n+3}
+\beta_{n}^5
+4\beta_{n}^4\beta_{n+1}
+6\beta_{n}^3\beta_{n+1}^2
+3\beta_{n}^3\beta_{n+1}\beta_{n+2}
\\&
+4\beta_{n}^2\beta_{n+1}^3
+6\beta_{n}^2\beta_{n+1}^2\beta_{n+2}
+2\beta_{n}^2\beta_{n+1}\beta_{n+2}^2
+2\beta_{n}^2\beta_{n+1}\beta_{n+2}\beta_{n+3}
+\beta_{n}\beta_{n+1}^4
+3\beta_{n}\beta_{n+1}^3\beta_{n+2}
\\&
+3\beta_{n}\beta_{n+1}^2\beta_{n+2}^2
+2\beta_{n}\beta_{n+1}^2\beta_{n+2}\beta_{n+3}
+\beta_{n}\beta_{n+1}\beta_{n+2}^3
+2\beta_{n}\beta_{n+1}\beta_{n+2}^2\beta_{n+3}
+\beta_{n}\beta_{n+1}\beta_{n+2}\beta_{n+3}^2
\\&
+\beta_{n}\beta_{n+1}\beta_{n+2}\beta_{n+3}\beta_{n+4})
+8t_4(\beta_{n-2}^2\beta_{n-1}\beta_{n}
+2\beta_{n-2}\beta_{n-1}^2\beta_{n}
+2\beta_{n-2}\beta_{n-1}\beta_{n}^2
+\beta_{n-2}\beta_{n-1}\beta_{n}\beta_{n+1}
\\&
+\beta_{n-3}\beta_{n-2}\beta_{n-1}\beta_{n}
+\beta_{n-1}^3\beta_{n}
+3\beta_{n-1}^2\beta_{n}^2
+\beta_{n-1}^2\beta_{n}\beta_{n+1}
+3\beta_{n-1}\beta_{n}^3
+4\beta_{n-1}\beta_{n}^2\beta_{n+1}
+3\beta_{n}^2\beta_{n+1}^2
\\&
+\beta_{n-1}\beta_{n}\beta_{n+1}^2+\beta_{n-1}\beta_{n}\beta_{n+1}\beta_{n+2}
+3\beta_{n}^3\beta_{n+1}
+\beta_{n}^4
+2\beta_{n}^2\beta_{n+1}\beta_{n+2}
+\beta_{n}\beta_{n+1}^3
+2\beta_{n}\beta_{n+1}^2\beta_{n+2}
\\&
+\beta_{n}\beta_{n+1}\beta_{n+2}^2
+\beta_{n}\beta_{n+1}\beta_{n+2}\beta_{n+3})+6t_3\beta_{n}(\beta_{n-2}\beta_{n-1}+\beta_{n-1}^2+2\beta_{n-1}\beta_{n}+\beta_{n-1}\beta_{n+1}+\beta_{n}^2
\\&
+2\beta_{n}\beta_{n+1}+\beta_{n+1}^2+\beta_{n+1}\beta_{n+2})+4t_2\beta_{n}(\beta_{n-1}+\beta_{n}+\beta_{n+1})+2t_1\beta_{n}-\left(\lambda+\frac{1}{2}\right)(1-(-1)^n)=n.
            \end{align*}
\end{remark}

For the purpose of the following discussion, we define
\begin{align}
    F_{m+s,n}=r_{n,m+s}= \beta_{n}\!\!\!\sum_{\substack{i+j=m+s-1\\0\le i,j \le m-1}}(F_{i,n}+F_{i,n+1})(F_{j,n}+F_{j,n-1})-\!\!\!\sum_{\substack{i+j=m+s\\1\le i,j \le m-1}}F_{i,n}F_{j,n}\label{F_{m+s,n}}
\end{align}
where $s=1,2,\cdots,m-1$. Then from \eqref{24.17}, we have the following theorem.

\begin{theorem}\label{The nontrivial leading coefficient}
The nontrivial leading coefficient $\mathrm{p}(n;T_m;\lambda)$ can be expressed by $F_{k,n}(\cdot)$ explicitly as follows:
       	\begin{align}
       		\mathrm{p}(n;T_m;\lambda)=&\sum_{k=1}^{m-1}kt_k[F_{1,n}(F_{k,n-1}+F_{k,n}+F_{k,n+1})-F_{k+1,n}]\notag\\
       		&+mt_m(F_{1,n}F_{m,n}-F_{m+1,n})-(\frac{n}{2}+\lambda)F_{1,n},\label{24.20}
       	\end{align}	
       	where $F_{k,n}(\cdot)$ are the same as those in Theorem \ref{thm-dPI} when $k=1,2,\cdots,m$ and $F_{m+s}, s=1,2,\cdots,m-1$ are defined in \eqref{F_{m+s,n}}.

\end{theorem}

\begin{proof}
From \eqref{9}, we get an important identity
       	\begin{equation}
       		\sum_{i=0}^{n-1}\beta_i=-\mathrm{p}(n;T_m;\lambda).\label{24.21}
       	\end{equation}
       	Noting that $\beta_{0}=0$, we further have
       \begin{equation}\label{eq:relation-betan-p}
       		\sum_{i=0}^{n-1}(\beta_{i}+\beta_{i+1})=\beta_{n}-2\mathrm{p}(n;T_m;\lambda).
       	\end{equation}
       It can be seen from \eqref{24.1712} and \eqref{24.16} that
       \begin{equation}
       	\sum_{k=1}^{m}2kt_kr_{n,k}-(2\lambda+1)r_{n,0}=n.\label{24.222}
       \end{equation}
Then, a combination of \eqref{24.222} and \eqref{24.17} yields
       \begin{align}
       	\sum_{j=0}^{n-1}R_{j,1}&=\sum_{k=1}^{m}2kt_kr_{n,k+1}-\sum_{k=1}^{m-1}2\beta_{n}kt_k(r_{n+1,k}+r_{n-1,k})+(2\lambda+1)(1-r_{n,0})r_{n,1}\notag\\
       	&=\sum_{k=1}^{m}2kt_kr_{n,k+1}-\sum_{k=1}^{m-1}2\beta_{n}kt_k(r_{n-1,k}+r_{n+1,k})-\sum_{k=1}^{m}2\beta_{n}kt_kr_{n,k}+(2\lambda+1+n)r_{n,1}.
       \end{align}
       On the other hand, according to \eqref{24} and \eqref{eq:relation-betan-p}, we have
       \begin{equation}
       \sum_{j=0}^{n-1}R_{j,1}=\beta_{n}-2\mathrm{p}(n;T_m;\lambda).\label{24.223}
       \end{equation}
       Hence, we get \eqref{24.20} by noting that $F_{1,n}=\beta_{n}$ and $F_{k,n}=r_{n,k}$.
       \end{proof}
       	
    \begin{remark}
    When $m=3$ and $t_{3}=1$ and $\lambda=-\frac{1}{2}$, \eqref{24.20} becomes
       		\begin{align}
       			&\mathrm{p}\left(n;T_3;-\frac{1}{2}\right)\notag\\
                =&\frac{1}{2}[\beta_{n}\!-n\beta_{n}-4t_2(r_{n,3}-\beta_{n}(r_{n-1,2}+r_{n,2}+r_{n+1,2}))-6t_3(r_{n,4}-\beta_{n}(r_{n-1,3}+r_{n,3}+r_{n+1,3}))]\notag\\
       			=&\frac{1}{2}\beta_{n}[1-n-4t_2 \beta_{n-1}\beta_{n+1}-6t_3\beta_{n-1}\beta_{n+1}(\beta_{n-2}+
       			\beta_{n-1}+\beta_{n}+\beta_{n+1}+\beta_{n+2})].\label{24.23}
       	\end{align}
It is consistent with the corresponding result in \cite{10}.

    \end{remark}
       	
     \section{Differential equations}\label{sec:differential-equations}
     When $m=2$ and $m=3$ with $\lambda=-\frac{1}{2}$, the second-order differential equations satisfied by the orthogonal polynomials $P_{n}(x; T_{m};\lambda)$ were studied in \cite{24} and \cite{23}, respectively. We find the corresponding results can be extended to general $m\in\mathbb{N}^{+}$ and $\lambda>-1$. Precisely, we have the following theorem.\\

   \begin{theorem}
  The monic orthogonal polynomials $P_n(x)$ satisfy the second-order differential equation:
     \begin{align}
     	&P''_n(x; T_{m};\lambda)-\left(V'(x)+\frac{A'_n(x)}{A_n(x)}\right)P'_n(x; T_{m};\lambda) \notag\\
        +&\left(B'_n(x)-B^2_n(x)-V'(x)B_n(x)
     	-\beta_{n}A_n(x)A_{n-1}(x)-\frac{A'_n(x)B_n(x)}{A_n(x)}\right)P_n(x; T_{m};\lambda)=0,\label{24.230001}
     	\end{align}
     where
     \begin{align*}
     	&V(x)=\sum_{k=1}^{m}t_kx^{2k}-(2\lambda+1)\ln \lvert x \rvert,\\
     	&A_n(x)=\sum_{k=2}^{m}\sum_{i=1}^{k-1}2kt_k(F_{i,n}+F_{i,n+1})x^{2(k-1-i)}+1,\\
     	&B_n(x)=\sum_{k=2}^{m}\sum_{i=1}^{k-1}2kt_kF_{i,n}x^{2(k-i-1)+1}+\frac{2\lambda+1}{2x}(1-(-1)^n).\label{24.230002}
     	\end{align*}
     $F_{i,n}$ is given by the iterations \eqref{eq:F-recurrence}, $\lambda>-1$ and $t_m>0$.

   \end{theorem}

   \begin{proof}
According to the operator \eqref{11}, we have
     \begin{equation}
     	P_{n-1}(x; T_{m};\lambda)=\frac{P'_n(x; T_{m};\lambda)+B_n(x)P_n(x; T_{m};\lambda)}{\beta_{n}A_n(x)}.\label{24.230003}
     \end{equation}
      Then differentiating both sides of the equation \eqref{11} with respect to $x$, we have
     \begin{equation}
     \begin{split}
     	P''_n(x; T_{m};\lambda)=&\beta_{n}(A'_n(x)P_{n-1}(x; T_{m};\lambda)+A_n(x)P'_{n-1}(x; T_{m};\lambda))\\
     	&-(B'_n(x)P_n(x; T_{m};\lambda)+B_n(x)P'_n(x; T_{m};\lambda)).\label{24.230004}
     \end{split}
     \end{equation}
     Substituting \eqref{12} and \eqref{24.230003} into \eqref{24.230004}, we can obtain \eqref{24.230001}. For $A_{n}(x)$ and $B_n(x)$, we only use the fact that $F_{i,n}=r_{n,i}$ and $R_{n,i}=r_{n,i}+r_{n+1,i}$ for $i=1,2,\cdots,m$.

   \end{proof}

Basing on the analysis in section 2, we can also obtain the differential identities satisfied by $\beta_{n}$.

\begin{theorem}\label{thm-partial-betan-tk}
For each $k=1,2,\cdots, m$,
we have
     \begin{align}
     	\frac{\partial \beta_{n}}{\partial t_k}&=F_{1,n}(F_{k,n-1}-F_{k,n+1}).\label{eq:Volterra}
     \end{align}
 where $F_{1, n}=\beta_{n}$ and $F_{k,n}:=F_{k,n}(\beta_{n-k+1},\beta_{n-k+2},\cdots,\beta_{n+k-2},\beta_{n+k-1}), k=2,\cdots,m$ are given by the iterations \eqref{eq:F-recurrence}.
 \end{theorem}

 \begin{proof}
     According to \eqref{24.01}, we have
 \begin{align}
 	h_n=\int_{-\infty}^{+\infty}P_n^2(x; T_{m};\lambda)w(x; T_{m};\lambda)dx.\label{24.230006}
 \end{align}
 Differentiating both sides of the equation \eqref{24.230006} with respect to $t_k$, we have
 \begin{equation}
 \begin{split}
 	\frac{\partial h_n}{\partial t_k}=&2\int_{-\infty}^{+\infty}P_n(x; T_{m};\lambda)\frac{\partial P_n(x; T_{m};\lambda)}{\partial t_k}w(x; T_{m};\lambda)dx\\
 	&+\int_{-\infty}^{+\infty}P_n^2(x; T_{m};\lambda)\frac{\partial w(x; T_{m};\lambda)}{\partial t_k}dx.\label{24.230007}
 \end{split}
 \end{equation}
 Applying $\frac{\partial }{\partial t_k}$ to $P_n(x)$ yields
 \begin{align*}
 	\frac{\partial P_n(x; T_{m};\lambda)}{\partial t_k}=\frac{\partial }{\partial t_k}\left(x^n+\mathrm{p}(n;T_m;\lambda)x^{n-2}+\cdots\right)=\frac{\partial }{\partial t_k}\mathrm{p}(n;T_m;\lambda)\cdot P_{n-2}(x; T_{m};\lambda)+\cdots.
 \end{align*}
 According to the orthogonality relation , the first integral on the right hand side of \eqref{24.230007} is zero. Noting that
 \begin{align*}
 	\frac{\partial w(x; T_{m};\lambda)}{\partial t_k}=-x^{2k}w(x; T_{m};\lambda),
 \end{align*}
 with aid of \eqref{22},hence \eqref{24.230007} becomes
 \begin{align}
 	\frac{\partial h_n}{\partial t_k}=-h_n R_{n,k}.\label{24.230008}
 \end{align}
Then differentiating both sides of the equation $h_n=\beta_{n}h_{n-1}$ with respect to $t_k$,we get
\begin{align}
	\frac{\partial h_n}{\partial t_k}=\frac{\partial \beta_{n}}{\partial t_k}h_{n-1}+\beta_{n}\frac{\partial h_{n-1}}{\partial t_k}.\label{24.230009}
\end{align}
Combining that \eqref{24.230008} and \eqref{24.230009},we have
\begin{align}
	\frac{\partial \beta_{n}}{\partial t_k}=\beta_{n}\left(R_{n-1,k}-R_{n,k}\right).\notag
\end{align}
According to \eqref{24.1712} and the fact $F_{k,n}=r_{n,k}$, we can obtain \eqref{eq:Volterra}.
\end{proof}

\begin{remark}
    \textit{Equation \eqref{eq:Volterra} is known as Volterra hierarchy which is appeared in several literature \cite{26, CMW-2024, 25}.  When $k=1$, equation \eqref{eq:Volterra} becomes
\begin{align*}
	\frac{\partial \beta_{n}}{\partial t_1}=\beta_{n}(\beta_{n-1}-\beta_{n+1}).
	\end{align*}
When $k=2$, equation \eqref{eq:Volterra} becomes
\begin{align*}
	\frac{\partial \beta_{n}}{\partial t_2}=\beta_{n}[\beta_{n-1}(\beta_{n}+\beta_{n-1}+\beta_{n-2})-\beta_{n+1}(\beta_{n}+\beta_{n+1}+\beta_{n+2})].
	\end{align*}
They are consistent with Lemma 2.4 in \cite{25}.}
\end{remark}

Combining the Theorems \ref{thm-dPI} and \ref{thm-partial-betan-tk}, we can obtain a partial differential equation satisfied by $\ln{\beta_{n}}$.
\begin{corollary}
    For any $m\in\mathbb{N}^{+}$, the recurrence coefficient $\beta_{n}$ defined in \eqref{three-term-recurrence} satisfies
    \begin{equation}
    \sum\limits_{k=1}^{m}2kt_{k}\frac{\partial \ln{\beta_{n}}}{\partial t_{k}}+2=0.
    \end{equation}
\end{corollary}

To the best of our knowledge, the above PDE seems to be new in the literature, and serves as a strong consistency check on our general discrete Painlev\'{e} hierarchy. It reveals a homogeneity of the recurrence coefficient with respect to the coupling constants $t_k$ that is independent of the singularity parameter $\lambda$.

Based on the relationship between $\beta_{n}$ and $h_{n}$ in \eqref{eq:relation-betan-hn}, and following the representation of the Hankel determinant in terms of $h_{n}$ given in \eqref{eq:relation-Dn-hn}, we further derive the difference-differential equations satisfied by $\ln h_{n}(T_{m}; \lambda)$ and $\ln D_{n}(T_{m}; \lambda)$.

\begin{corollary}
    For any $m, n\in\mathbb{N}^{+}$ and $n>1$, the normalization constants $h_{n}$ defined in \eqref{eq:def-hn} and the Hankel determinants $D_{n}(T_{m}; \lambda)$ given in \eqref{eq:def-Hankel-determinant} satisfy
    \begin{align}
&\sum\limits_{k=1}^{m}2kt_{k}\frac{\partial \ln h_{n}(T_{m}; \lambda)}{\partial t_{k}}+2(n-1)=0,\\
&\sum\limits_{k=1}^{m}2kt_{k}\frac{\partial \ln D_{n}(T_{m}; \lambda)}{\partial t_{k}}+n(n-1)=0.
    \end{align}
\end{corollary}

The partial derivatives of the logarithm of the Hankel determinant with respect to $t_{s}, s=1,2,\cdots,m-1$ can also be explicitly represented by $\beta_{n}$. Precisely we have the following results.

\begin{theorem}\label{hankel Fk expression}
For each $s=1,2,\cdots, m-1$, the quantity $\frac{\partial}{\partial t_s}\ln D_n(T_m;\lambda)$ can be expressed by $F_{k,n}(\cdot)$ as follows:
     	\begin{align}
     		\frac{\partial}{\partial t_s}\ln D_n&(T_m;\lambda)=\sum_{k=1}^{m-1}\!\!\sum_{\substack{i=0\\k+i\le m-1}}^{s-1}2\beta_{n}kt_k((F_{k+i,n}+F_{k+i,n+1})(F_{s-i-1,n-1}+F_{s-i-1,n}))\notag\\
&+\sum_{k=1}^{m-1}\sum_{\substack{i=0\\k+i\le m-1}}^{s-1}2\beta_{n}kt_k\left((F_{k+i,n-1}+F_{k+i,n})(F_{s-i-1,n}+F_{s-i-1,n+1})\right)\notag\\
            &-\sum_{k=1}^{m-1}\sum_{\substack{i=0\\k+i\le m-1}}^{s-1}4kt_kF_{k+i,n}F_{s-i,n}-\sum_{k=1}^{m}2kt_kF_{k+s,n}
     		-(\lambda+\frac{1}{2})\left(1+(-1)^n\right)F_{s,n}.\label{24.230010}
     	\end{align}
     	where $F_{0,n}=\frac{1-(-1)^n}{2}$,  $F_{k,n}(\cdot)$ are the same as those in Theorem \ref{thm-dPI} when $k=1,2,\cdots,m$ and $F_{m+s}, s=1,2,\cdots,m-1$ are defined in \eqref{F_{m+s,n}}.
        \end{theorem}
     	
     	\begin{proof}
     	    From \eqref{24.230008} we have
     	\begin{equation}
     		\frac{\partial}{\partial t_k}\ln h_n(T_m;\lambda)=-R_{n,k} ,\notag
     	\end{equation}
     	It follows from \eqref{eq:relation-Dn-hn} that
     	\begin{equation}
     		\frac{\partial}{\partial t_k}\ln D_n(T_m;\lambda)=-\sum_{i=0}^{n-1}R_{i,k} .\label{24.230011}
     	\end{equation}
     	From \eqref{24.1704}, we can get
     	\begin{align}
     	\sum_{j=0}^{n-1}R_{j,s}=&\sum_{k=1}^{m-1}\!\!\sum_{\substack{i=0\\k+i\le m-1}}^{s-1}\left(4kt_kr_{n,k+i}r_{n,s-i}-2\beta_{n}kt_k\left(R_{n,k+i}R_{n-1,s-i-1}+R_{n,s-i-1}R_{n-1,k+i}\right)\right)\notag\\
     	&+\sum_{k=1}^{m}2kt_kr_{n,k+s}+(2\lambda+1)(1-r_{n,0})r_{n,s} .\label{24.230012}
     	\end{align}
     	where $s=1 , 2 , \cdots , m-1$ , with the aid of \eqref{24.1712} and \eqref{eq:F-recurrence} we can obtain the desired result.
        \end{proof}
     	
     	\begin{remark}
        When $m=3 , t_3=1 , \lambda=-\frac{1}{2}$, the weight \eqref{eq:Freud-weight} becomes
     	\begin{align*}
     		w\left(x;1,t_2,t_1;-\frac{1}{2}\right)=exp(-x^6-t_2x^4-t_1x^2),
     		\end{align*}
     	According to \eqref{24.230010}, we have
     	\begin{align}
     		\frac{\partial}{\partial t_1}\ln &D_n\left(1,t_2,t_1;-\frac{1}{2}\right)=-2t_1\beta_n^2-4t_2\beta_{n}(\beta_{n-1}+\beta_{n})(\beta_{n}+\beta_{n+1})-6\beta_{n}[\beta_{n}^2(\beta_{n-1}+\beta_{n}+\beta_{n+1})\notag\\
     		&+(\beta_{n}+\beta_{n+1})\beta_{n-1}(\beta_{n-2}+\beta_{n-1}+\beta_{n})+(\beta_{n-1}+\beta_{n})\beta_{n+1}(\beta_{n+2}+\beta_{n+1}+\beta_{n})] .\notag
     	\end{align}	
     	\begin{align*}
     		\frac{\partial}{\partial t_2}\ln &D_n\left(1,t_2,t_1;-\frac{1}{2}\right)=2t_1(\beta_{n-1}\beta_{n}\beta_{n+1}-\beta_{n}^2(\beta_{n-1}+\beta_{n}+\beta_{n+1}))-4t_2\beta_{n}^2(\beta_{n-1}+\beta_{n}+\beta_{n+1})^2\\
     		&-6\beta_{n}[\beta_{n-1}(\beta_{n-2}+\beta_{n-1}+\beta_{n})+\beta_{n}(\beta_{n-1}+\beta_{n}+\beta_{n+1})]\times[\beta_{n}(\beta_{n-1}+\beta_{n}+\beta_{n+1})\\
     		&+\beta_{n+1}(\beta_{n+2}+\beta_{n+1}+\beta_{n})] .
     		\end{align*}
     	The above results are consistent with the ones in \cite[Lemma 3.3]{10} .

      \end{remark}

\section{Large $n$ asymptotics}
\label{sec: Large-n-asym} 

When $\lambda=-1/2$, Deift et al. \cite{DKMVZ-1999} conclude that, as $n\to\infty$, 
\begin{equation}\label{eq:full-asym-beta-n-Deift}
    \beta_{n}\sim \left(\frac{n}{4^mA_m}\right)^{\frac{1}{m}}\left(1+\sum\limits_{i=1}^{\infty}\frac{a_{i}}{n^{\frac{i}{m}}}\right).
\end{equation}
by analyzing the asymptotic behavior of the corresponding orthogonal polynomials using the Riemann-Hilbert method. 
The leading coefficient is explicitly given and the sub-leading ones can be derived by investigating further asymptotic analysis following the Riemann-Hilbert analysis in \cite{DKMVZ-1999} but rather complicated. In this section, we show that the first $m-1$ coefficients $a_{i}, k=1,2,\cdots,m-1$ can be 
derived in a simpler way using the discrete Painlev\'{e} I hierarchy \eqref{eq:definition-dPI}.     
     
\begin{lemma}\label{lem:asy-beta-finite}
         When $\lambda=-\frac{1}{2}$, define 
         \begin{equation}\label{eq:value-Ak}
             A_k=\frac{t_k\Gamma(k+\frac{1}{2})}{\Gamma(\frac{1}{2})\Gamma(k)}, \quad k=1,2,\cdots,m.
         \end{equation}
         and let $u=u(n)$ be the unique positive solution of
     \begin{equation}\label{eq:equation-u}
      	\sum_{k=1}^{m}A_ku^{k}=n, 
     \end{equation}
then the recurrence coefficient $\beta_{n}$ satisfies
     \begin{align}\label{eq:asym-beta-u}
     	\beta_{n}=\frac{u}{4}\left(1+\mathcal{O}\left(\frac{1}{n}\right)\right),\qquad n\to\infty
     \end{align}
     where $a_{i}, i=1,2,\cdots, m-2$ can be recursively using \eqref{eq:equation-u}.
     \end{lemma}

Before proving the above lemma, we first show that
\begin{equation}\label{eq:leading-asym-Fkn}
    F_{k,n}\sim N_{k}n^{\frac{k}{m}}, \qquad n\to\infty,
\end{equation}
     where        
     \begin{equation}
     	N_1 = \frac{1}{4}\left(\frac{\Gamma(\frac{1}{2})\Gamma(m)}{t_m\Gamma(m+\frac{1}{2})}\right)^{\frac{1}{m}}, \label{eq:explicit-value-N1}
     \end{equation}
     and 
         \begin{align}\label{eq:explicit-value-Nk}
  	N_k=\frac{1}{2}\binom{2k}{k}N_{1}^{k},\quad k=1,2,\cdots,m.
     \end{align}

Since $F_{1,n}=\beta_{n}$, we can check \eqref{eq:explicit-value-N1} directly from \eqref{eq:full-asym-beta-n-Deift}. When $1<k\leq 2m-1$, by substituting \eqref{eq:leading-asym-Fkn} into \eqref{eq:F-recurrence}, we have
\begin{align}
N_k=\begin{cases}4N_1\left(N_{k-1}+\sum\limits_{\substack{i+j=k-1\\1\le i,j \le k-2}}N_iN_j\right)-\sum\limits_{\substack{i+j=k\\1\le i,j \le k-1}}N_iN_j, \quad k=1,2,\cdots,m,\\
4N_1\sum\limits_{\substack{i+j=k-1\\1\le i,j \le m-1}}N_iN_j-\sum\limits_{\substack{i+j=k\\1\le i,j \le m-1}}N_iN_j, \quad k=m+1,\cdots,2m-1. \end{cases}\label{eq:recurrence-Nk}
     \end{align} 
Now we will further prove
    \begin{align}\label{eq:explicit-value-Nk}
  	N_k=\frac{1}{2}\binom{2k}{k}N_{1}^{k},\quad k=1,2,\cdots,m, 
     \end{align}
by mathematical induction. When $k =1$, it is obviously correct.
Assuming that \eqref{eq:explicit-value-Nk} holds for $k=1,2,\cdots, p-1$, where $1<p<m$ is a positive integer. we need to show that the \eqref{eq:explicit-value-Nk} holds for $k = p$, {\it i.e.} to prove
     \begin{equation}
         \frac{1}{2}\binom{2p}{p}=2\binom{2p-2}{p-1}+\sum_{\substack{i+j=p-1\\1\le i,j \le p-2}}\binom{2i}{i}\binom{2j}{j}-\frac{1}{4}\sum_{\substack{i+j=p\\1\le i,j \le p-1}}\binom{2i}{i}\binom{2j}{j}.\label{equation N}
     \end{equation}
     Consider the Maclaurin series of the function $(1-4x)^{-\frac{1}{2}}$
     \begin{align*}
         (1-4x)^{-\frac{1}{2}}=\sum_{i=0}^{\infty}\binom{2i}{i}x^i,\qquad |x|<\frac{1}{4}.
     \end{align*}
   Squaring both sides of the above equation, we have
     \begin{align*}
         \left(\sum_{i=0}^{\infty}\binom{2i}{i}x^i\right)^2=\frac{1}{1-4x}=\sum_{i=0}^{\infty}4^i x^i.
     \end{align*}
By comparing the coefficients on both sides, we find that
     \begin{equation}
         \sum_{i=0}^{p}\binom{2i}{i}\binom{2(p-i)}{p-i}=4^p.
     \end{equation}
It follows that
     \begin{equation}
         2\binom{2p}{p}+\sum_{\substack{i+j=p\\1\le i,j \le p-1}}\binom{2i}{i}\binom{2j}{j}=4^{p}.\label{series 1}
     \end{equation}
Replacing  $p$ by $p-1$ in the above equation, we have
     \begin{equation}
         2\binom{2p-2}{p-1}+\sum_{\substack{i+j=p-1\\1\le i,j \le p-2}}\binom{2i}{i}\binom{2j}{j}=4^{p-1}.\label{series 2}
     \end{equation}
A combination of \eqref{series 1} and  \eqref{series 2} yields \eqref{equation N}.
Therefore, by mathematical induction, we show that \eqref{eq:explicit-value-Nk} holds for all $k=1,2,\cdots,m$.

\textbf{Proof of Lemma \ref{lem:asy-beta-finite}. }
    From \eqref{eq:full-asym-beta-n-Deift}, we find that $\beta_{n-k}=\beta_{n}\left(1+\mathcal{O}\left(\frac{1}{n}\right)\right)$. Hence, a combination of the discrete Painlev\'{e} I hierarchy \eqref{eq:definition-dPI} and \eqref{eq:leading-asym-Fkn} yields
    \begin{equation}\label{eq:asy-beta-equaiton}
\sum\limits_{k=1}^{m}A_{k}(4\beta_{n})^{k}=n+\mathcal{O}(1),
    \end{equation}
    where 
    \begin{equation}
        A_{k}=2^{-2k+1}kt_{k}N_{k}(N_{1})^{-k}, \quad k=1,2,\cdots,m.
    \end{equation}
 Making use of the explicit values of $N_{k}$ in \eqref{eq:explicit-value-Nk}, we arrived at \eqref{eq:value-Ak}. Finally, when $n\to\infty$, by comparing \eqref{eq:equation-u} and \eqref{eq:asy-beta-equaiton}, we obtain \eqref{eq:asym-beta-u}. 
\qed

     \begin{remark}\label{rem-coefficients}
        Lemma \ref{lem:asy-beta-finite} gives a simple algorithm to derive the coefficients $a_{i}, i=1,2,\cdots,m-1$ in \eqref{eq:full-asym-beta-n-Deift}. From the combination of \eqref{eq:equation-u}, \eqref{eq:value-Ak} and \eqref{eq:explicit-value-N1}, we conclude that, as $n\to\infty$, the unique positive solution of \eqref{eq:equation-u} behaves like
      \begin{align}
      	u \sim (4N_{1})n^{\frac{1}{m}}\left(1+\sum_{i=1}^{\infty}\frac{a_i}{n^{\frac{i}{m}}}\right).\label{47.1}
      \end{align}
      Substituting \eqref{47.1} into \eqref{eq:equation-u} and comparing the coefficients of $n^{\frac{i}{m}}$ on both sides of the equation, we obtain a sequence of equations satisfied by coefficients $a_{i}$ and $A_m$:
      \begin{align}
      	n:&\quad A_m(4N_{1})^m=1.\label{47.2}\\
      	n^{1-\frac{1}{m}}:&\quad A_m(4N_{1})^mB_{1,m}+A_{m-1}(4N_{1})^{m-1}B_{0,m-1}=0.\label{47.3}\\
      	n^{1-\frac{2}{m}}:&\quad A_m(4N_{1})^mB_{2,m}+A_{m-1}(4N_{1})^{m-1}B_{1,m-1}+A_{m-2}(4N_{1})^{m-2}B_{0,m-2}=0.\label{47.4}\\
      	&\quad \cdots\notag\\
      	n^{1-\frac{k}{m}}:&\quad A_m(4N_{1})^m B_{k,m}+A_{m-1}(4N_{1})^{m-1}B_{k-1,m-1}+\cdots+A_{m-k}(4N_{1})^{m-k}B_{0,m-k}=0.\label{47.41}\\
        &\quad \cdots\notag\\
      	n^{\frac{2}{m}}:&\quad A_m(4N_{1})^mB_{m-2,m}+A_{m-1}(4N_{1})^{m-1}B_{m-3,m-1}+\cdots+A_{2}B_{0,2}=0,\label{47.411}
      \end{align}
      where $B_{0,0}=1, a_0=1$ and
      \begin{align*}
      	B_{k,j}=ja_{k}+\sum_{\substack{i_1+i_2+\cdots+i_j=k\\0\leq i_s< k}}a_{i_1}a_{i_2}\cdots a_{i_j},\qquad k=1,2,\cdots,m-2.
      \end{align*}
    From the above equation, we can compute $a_i$ recursively.
      \end{remark}

    When $\lambda>-1$, although  there is a Fisher-Hartwig singularity at $x=0$ in the weight \eqref{eq:Freud-weight}, it does not influence the equilibrium measure of the orthogonal polynomials. Hence the leading asymptotic behaivor of $\beta_{n}$ as $n\to\infty$ keep the same as that in \eqref{eq:full-asym-beta-n-Deift}. It means that we still have 
    \begin{equation}\label{eq:leading-asy-beta-general}
        \beta_{n}:=\beta_{n}(\lambda)\sim N_{1}n^{\frac{1}{m}}, \quad n\to\infty.
    \end{equation}
    for any fixed $\lambda>-1$. Making use of this fact and the explicit representations of $\mathrm{p}(n;T_{m};\lambda)$ and $\frac{\partial}{\partial t_{s}}\ln D_{n}(T_{m};\lambda)$ by $\beta_{n}$ in \eqref{24.20} and \eqref{24.230010}, we can obtain the leading asymptotic behaivor of $\mathrm{p}(n;T_{m};\lambda)$ and $\ln D_{n}(T_{m};\lambda)$ as $n\to\infty$.

     \begin{theorem}\label{thm-p-leading-behavior}
       For the generalized Freud weight \eqref{eq:Freud-weight}, the nontrivial leading coefficient $\mathrm{p}(n;T_m;\lambda)$ of the monic orthogonal polynomials $P_{n}^{(m)}(x)$ satisfies
     \begin{align}
     	\mathrm{p}(n;T_m;\lambda)\sim -\frac{mN_{1}}{(m+1)}n^{\frac{m+1}{m}} \label{49.01}
     	\end{align}
as $n\to\infty$, where $N_{1}$ is given in \eqref{eq:explicit-value-N1}.
     \end{theorem}

     \begin{proof}
     According to the definition of $F_{k,n}$ in \eqref{eq:F-recurrence} and the asymptotic behavior of $\beta_{n}$ in \eqref{eq:leading-asy-beta-general}, we conclude that $F_{k,n}$ has the following asymptotic behavior
     \begin{align}
     F_{k,n}\sim N_kn^{\frac{k}{m}},\qquad k=1,2,\cdots,2m-1,
   \label{eq:Fk-leading-behavior}
     \end{align} 
where
\begin{align}
N_k=\begin{cases}4N_1\left(N_{k-1}+\sum\limits_{\substack{i+j=k-1\\1\le i,j \le k-2}}N_iN_j\right)-\sum\limits_{\substack{i+j=k\\1\le i,j \le k-1}}N_iN_j, \quad k=1,2,\cdots,m,\\
4N_1\sum\limits_{\substack{i+j=k-1\\1\le i,j \le m-1}}N_iN_j-\sum\limits_{\substack{i+j=k\\1\le i,j \le m-1}}N_iN_j, \quad k=m+1,\cdots,2m-1. \end{cases}\label{eq:recurrence-Nk-2}
     \end{align} 
and we have proved 
    \begin{align}\label{eq:explicit-value-Nk}
  	N_k=\frac{1}{2}\binom{2k}{k}N_{1}^{k},\quad k=1,2,\cdots,m
     \end{align}
by mathematical induction before. 

Combining \eqref{eq:explicit-value-N1} and \eqref{eq:explicit-value-Nk}, we find that     \begin{equation}
        2mt_{m}N_{m}=1.
    \end{equation}    
Moreover, combining \eqref{eq:recurrence-Nk} and \eqref{equation N}, we have
    \begin{equation}
    \begin{split}
        N_{m+1}=&\frac{1}{2}\binom{2m+2}{m+1}N_{1}^{m+1}-2N_{1}N_{m}
        \end{split}
    \end{equation}
Making use of these facts, we obtain \eqref{49.01} by substituting \eqref{eq:Fk-leading-behavior} into \eqref{24.20}.
     \end{proof}

     \begin{remark}
        It should also be noted that, from the definition of $N_{k}$ in \eqref{eq:recurrence-Nk-2}, $N_{k}\neq \frac{1}{2}\binom{2k}{k}N_1^k$ when $k\ge m+1$.
     \end{remark}

     \begin{theorem}
         As $n\to \infty$, The Hankel determinant $D_n(t_m,t_{m-1},\cdots,t_s,\cdots ,t_1;\lambda)$ satisfies
      \begin{align}
      \ln\frac{D_n(t_m,\cdots,t_{s+1},t_s,t_{s-1},\cdots ,t_1;\lambda)}{D_n(t_m,\cdots,t_{s+1},0,t_{s-1},\cdots ,t_1;\lambda)}\sim-2mt_{s}t_mN_{m+s}n^{\frac{m+s}{m}}.\label{53.2}
      	\end{align}
      where $N_{k}$ is defined by \eqref{eq:recurrence-Nk}, $s=1,2,\cdots,m-1$ and $d_{s,i}$ are constant independent of $n$.\end{theorem}

       \begin{proof}
          Combining \eqref{24.230010} and \eqref{eq:Fk-leading-behavior}, we have
       \begin{align}
       	&\frac{\partial}{\partial t_s}\ln D_n(t_m,\cdots,t_{s},\cdots,t_1;\lambda)\sim-2mt_mN_{m+s}n^{\frac{m+s}{m}}  \label{53.21}
       \end{align}
       as $n\to\infty$.
        Integrating both sides with respect to $t_s$, we get \eqref{53.2}.
       \end{proof}

\section{Full asymptotic expansions as $n\to\infty$}
In previous section, we have only obtained the leading behavior of $\beta_{n}$, $\mathrm{p}(n; T_{m},\lambda)$ and $\ln{D_{n}(T_{m};\lambda)}$ for general $\lambda>-1$. With the aid of the discrete Painlev\'{e} I hierarchy, one may derive a formal asymptotic expansion of $\beta_{n}$ as $n\to\infty$. Some special cases are given in literature without rigorous proof on the error approximation.
In this section, combining the discrete Painlev\'{e} equation and the asymptotic theory for the linear difference equation \cite{Elaydi-difference}, we proceed to investigate the full asymptotic expansion of $\beta_{n}$, $\mathrm{p}(n; T_{m},\lambda)$ and $\ln{D_{n}(T_{m};\lambda)}$ as $n\to\infty$.

In Section \ref{sec:dPI} and Section \ref{sec:differential-equations}, we see that the recurrence coefficients $\beta_n$, the nontrivial leading coefficients $\mathrm{p}(n;T_m,\lambda)$, and the Hankel determinant $\ln D_n(T_m,\lambda)$ are all expressed in terms of the auxiliary quantity $F_{k,n}$. Hence, we first show the following theorem.
\begin{theorem}\label{thm:full-asym-Fkn}
    For any $p\in\mathbb{N}^{+}$ and $k=1,2,\dots,m$, let 
    \begin{equation}\label{eq:Delta-n}
       \Delta_{n}=\frac{2\lambda+1}{2}\bigl(1-(-1)^n\bigr),
    \end{equation}
    and denote $F_{k,n}$ by
    \begin{equation}\label{eq:full-asym-Fkn}
        F_{k,n}=N_{k}n^{\frac{k}{m}}\Bigg(\sum\limits_{s=0}^{p}\frac{a_{s,k}+b_{s,k}(\Delta_{n})}{n^{\frac{s}{m}}}+n^{-\frac{p+1}{m}}E_{p,k}(n)\Bigg),
    \end{equation}
    where the constants $N_{k}$ are defined in Theorem~\ref{thm-p-leading-behavior}, and the coefficients $a_{s,k}$ and $b_{s,k}(\Delta_{n})$ are determined recursively by substituting \eqref{eq:full-asym-Fkn} into the discrete Painlev\'{e} I hierarchy \eqref{eq:definition-dPI}. Then the error terms satisfy
    \begin{equation}
       \label{eq:error-full-asym-Fkn}
       E_{p,k}(n)=\mathcal{O}(1), \qquad n\to\infty.
    \end{equation}
    Moreover, the following properties hold:
    \begin{itemize}
        \item [(i)] $a_{s,k}$ are constants independent of $n$, with $a_{1,k}=1$ for all $k=1,2,\dots,m$;
        \item [(ii)] $a_{s,1}=a_{s}$ for $s=1,2,\cdots,m-1$, where $a_{s}$'s are given in \eqref{eq:full-asym-beta-n-Deift};
        \item [(iii)] $b_{s,k}(\Delta_{n})$ are polynomials in $\Delta_{n}$ such that $b_{s,k}(0)=0$ for all $s\in\mathbb{N}^{+}$, $k=1,2,\dots,m$;
        \item [(iv)] $b_{s,k}(\Delta_{n})=0$ for $s=1,2,\dots,m-1$ and all $k=1,2,\dots,m$.
    \end{itemize}
\end{theorem}

\begin{proof}
We first show that $a_{s,k}$ and $b_{s,k}(\Delta_{n}), s\in\mathbb{N}^{+}, k=1,2,\cdots,m$ can be determined recursively by the discrete Painlev\'{e} I equation \eqref{eq:definition-dPI}. 

When $p\leq m$, since $\beta_{n+k}=\beta_{n}\left(1+\mathcal{O}(n^{-1})\right)$ as $n\to\infty$, we find that the discrete Painlev\'{e} hierarchy \eqref{eq:definition-dPI} reduces to 
\begin{equation}\label{eq:beta-equaiton-leading-general}
    \sum\limits_{k=1}^{m}A_{k}(4\beta_{n})^{k}=n\left(1+\mathcal{O}(n^{-1})\right)
\end{equation}
as $n\to\infty$ uniformly for all fixed $\lambda>-1$. Here $\Delta_{n}$ is contained in the $\mathcal{O}(n^{-1})$ error term.
Comparing \eqref{eq:beta-equaiton-leading-general} and \eqref{eq:equation-u}, we find that $\Delta_{n}$ would only appear in the higher order terms of the asymptotic expansion of $\beta_{n}$. The leading terms before $\mathcal{O}(n^{-1})$ will keep the same with the special case $\lambda=-\frac{1}{2}$. Hence $a_{s,1}=a_{s}$ and $b_{s,1}(\Delta_{n})=0$ for all $s=1,2,\cdots,m-1$, where $a_{s}$ are given in \eqref{eq:full-asym-beta-n-Deift}. 
In particular, from \eqref{eq:Fk-leading-behavior}, we find that $a_{0,1}=1$. Substituting the asymptotic behavior of $F_{1,n}$ into \eqref{eq:F-recurrence}, we obtain that $a_{s,k}, s=1,2,\cdots,m-1, k=2,\cdots,m$ are also constants independent of $n$ and $b_{s,k}(\Delta_{n})=0$ for all $s=1,2,\cdots,m-1$ and $k=2,\cdots,m$.

When $p\geq m$, we assume that $a_{s,k}, b_{s,k}(\Delta_{n}), s=1,2,\cdots,p-1, k=1,2\cdots,m$ are obtained and intend to show that $a_{p,k}, b_{p,k}(\Delta_{n}), k=1,2,\cdots,m$ can be represented by $a_{1,1},\cdots, a_{p-1,1}$ and $b_{1,1}(\Delta_{n}),\cdots, b_{p-1,1}(\Delta_{n})$.

Substituting \eqref{eq:full-asym-Fkn} into \eqref{eq:definition-dPI}, we have
\begin{equation}\label{eq:recurrence-equations}
\begin{cases}
    & 2m t_{m}N_{m}a_{0,m}=1,\\
    & 2mt_{m}N_{m}a_{1,m}+2(m-1)t_{m-1}N_{m-1}a_{0,m-1}=0,\\
    & 2mt_{m}N_{m}a_{2,m}+2(m-1)t_{m-1}N_{m-1}a_{1,m-1}+2(m-2)t_{m-2}N_{m-2}a_{0,m-2}=0,\\
    &\cdots\\
    &2mt_{m}N_{m}a_{m-1,m}+2(m-1)t_{m-1}N_{m-1}a_{m-2,m-1}+\cdots+2t_{1}N_{1}a_{0,1}=0,\\
    &2mt_{m}N_{m}(a_{m,m}+b_{m,m}(\Delta_{n}))+2(m-1)t_{m-1}N_{m-1}a_{m-1,m-1}+\cdots+2t_{1}N_{1}a_{1,1}=\Delta_{n},\\
        &\cdots\\
        &2mt_{m}N_{m}(a_{p,m}+b_{p,m}(\Delta_{n}))+2(m-1)t_{m-1}N_{m-1}(a_{p-1,m-1}+b_{p-1,m-1}(\Delta_{n}))+\cdots\\
        &\qquad +2t_{1}N_{1}(a_{p-m+1,1}+b_{p-m+1,1}(\Delta_{n}))=0.
\end{cases}
\end{equation}
According to the results in Theorem \ref{thm-dPI}, we see that $F_{k,n}$ are polynomials of
$$\beta_{n-k+1},\beta_{n-k+2},\cdots, \beta_{n+k-1}$$
with non-negative coefficients. Hence we conclude that, for any $s=1,2,\cdots,p$ and $k=1,2,\cdots,m$, there exists a constant $\alpha_{s,k}>0$ and two functions $f_{s,k}, g_{s,k}$ such that
\begin{equation}\label{eq:ask-by-as1-fsk}
\begin{split}
    a_{s,k}+b_{s,k}(\Delta_{n})=&\alpha_{s,k}(a_{s,1}+b_{s,1}(\Delta_{n}))+f_{s,k}(a_{1,1},a_{2,1},\cdots,a_{s-1,1})\\
    &+g_{s,k}(b_{1,1}(\Delta_{n}),\cdots,b_{s-1,1}(\Delta_{n})).
    \end{split}
\end{equation}
Moreover, we find that $f_{s,k}$ and $g_{s,k}$ are all polynomials of the corresponding variables listed in the above formula and $g_{s,k}=0$ when $\Delta_{n}=0$.
Substituting it into the final equation of \eqref{eq:recurrence-equations}, we have
\begin{equation}
\begin{split}
&2mt_{m}N_{m}\alpha_{p,m}(a_{p,1}+b_{p,1}(\Delta_{n}))+\sum\limits_{k=1}^{m-1}2kt_{k}N_{k}\alpha_{s,k}a_{s,1}+\sum\limits_{k=1}^{m-1}2kt_{k}N_{k}\alpha_{s,k}b_{s,1}(\Delta_{n}))\\
&+\sum\limits_{k=1}^{m}2kt_{k}N_{k}f_{s,k}(a_{1,1},\cdots,a_{s-1,1})+\sum\limits_{k=1}^{m}2kt_{k}N_{k}g_{s,k}(b_{1,1}(\Delta_{n}),\cdots,b_{s-1,1}(\Delta_{n}))=0,
\end{split}
\end{equation}
where $s=p-m+k$.
This implies that $a_{p,1}$ can be explicitly represented by $a_{1,1}, a_{2,1},\cdots,a_{p-1,1}$ and $b_{p,1}(\Delta_{n})$ can be explicitly determined by $b_{1,1}(\Delta_{n}), b_{2,1}(\Delta_{n}),\cdots,b_{p-1,1}(\Delta_{n})$. A combination of this fact and \eqref{eq:ask-by-as1-fsk}, we conclude that all $a_{p,k}, k=2,3,\cdots,m$ can represented by $a_{1,1}, a_{2,1},\cdots,a_{p-1,1}$ and all $b_{p,k}(\Delta_{n}), k=2,\cdots,m$ can be explicitly determined by $b_{1,1}(\Delta_{n})$, $b_{2,1}(\Delta_{n})$, $\cdots$, $b_{p-1,1}(\Delta_{n})$.

Next, we prove \eqref{eq:error-full-asym-Fkn} by induction of $p$.  When $p=1$, it can be yield directly from \eqref{eq:leading-asy-beta-general} that $E_{1,1}(n)=\mathcal{O}(1)$
as $n\to\infty$. Substituting it into the iterative formula \eqref{eq:F-recurrence}, we find that \eqref{eq:error-full-asym-Fkn} holds for $p=1$ and all $k=1,2,\cdots$.

Assuming that \eqref{eq:error-full-asym-Fkn} holds for $p\leq M$ uniformly for $k=1,2,\cdots, m$, we now prove that it is also true for $p=M+1$. Substituting \eqref{eq:error-full-asym-Fkn} into \eqref{eq:definition-dPI} , we have
\begin{equation}\label{eq:recurrence-from-PI}
    2mt_mN_mE_{M+1,m}(n)+2(m-1)t_{m-1}N_{m-1}E_{M,m-1}(n)+\cdots+2t_1N_1E_{M-m+2,1}(n)=0.
\end{equation}
According to the inductive assumption that $E_{p,k}(n)=\mathcal{O}(1)$ for all $p\le M$ and $k=1,2,\cdots,m$, then \eqref{eq:recurrence-from-PI} implies
\begin{equation}
    E_{M+1,m}(n)=\mathcal{O}(1) .
\end{equation}
By the definition of $F_{m.n}$ in \eqref{eq:F-recurrence} and \eqref{eq:full-asym-Fkn}, we have
\begin{equation}\label{eq:difference-E(M+1)}
\begin{split}
    E_{M+1.m}(n)=&\sum\limits_{s=-m+1}^{m-1}\alpha_{s}E_{M+1,1}(n+s)+f_{M+1,m}(a_{1,1},\cdots,a_{M+1,1})\\
    &+g_{M+1,m}(b_{1,1}(\Delta_{n}),\cdots,b_{M+1,1}(\Delta_{n}))\\
    &+\sum\limits_{k=1}^{m}n^{-\frac{k}{m}}g_{k}(E_{M+1,1}(n-m+1),\cdots,E_{M+1,1}(n+m-1)).
    \end{split}
\end{equation}
Here, $\alpha_{-m+1},\alpha_{-m+2},\cdots,\alpha_{m-1}$ are all positive constants, $f_{M+1,m}, g_{M+1,m}$ are given in \eqref{eq:ask-by-as1-fsk} and $g_{k}$ is a polynomial of $$E_{M+1,1}(n-m+1),\cdots,E_{M+1,1}(n),\cdots,E_{M+1,1}(n+m-1)$$ with $\deg g_{k}=k$ and all coefficients being bounded. By the inductive assumption, we know that $E_{M+1,1}(n)=\mathcal{O}(n^{\frac{1}{m}})$ as $n\to\infty$. This means that
\begin{equation}
\alpha_{-m+1}E_{M+1,1}(n-m+1)+\cdots+\alpha_0E_{M+1,1}(n)+\cdots+\alpha_{m-1}E_{M+1,1}(n+m-1)=\mathcal{O}(1),
\end{equation}
as $n\to\infty$. Then by the standard theory of the linear difference equation with constant coefficients \cite{Elaydi-difference}, we conclude that
\begin{equation}\label{eq:bounded-E(M+1)}
    E_{M+1,1}(n)=\mathcal{O}(1), \quad n\to\infty.
\end{equation}
Combining this approximation with \eqref{eq:full-asym-Fkn} and making use of the iteration formula \eqref{eq:F-recurrence}, we get \eqref{eq:error-full-asym-Fkn} for all $k=1,2,\cdots,m$.

This complete the proof of Theorem \ref{thm:full-asym-Fkn}.
\end{proof}

From Theorem \ref{The nontrivial leading coefficient} and Theorem \ref{hankel Fk expression}, it can be seen that both $\mathrm{p}(n;T_m;\lambda)$ and $\frac{\partial}{\partial t_s}\ln D_n(T_m;\lambda)$ can be expressed in terms of $F_{k,n}$ .This means that their full asymptotic expansion as $n\to\infty$ can be obtained from Theorem \ref{thm:full-asym-Fkn}. Precisely, we have the following corollaries.
\begin{corollary}
The nontrivial leading coefficient $\mathrm{p}(n;T_m;\lambda)$ has the following full asymptotic expansion
\begin{equation}
    \mathrm{p}(n;T_m;\lambda) \sim -\frac{mN_{1}}{(m+1)}\sum_{k=0}^{\infty}u_{k}{n^{\frac{m+1-k}{m}}},\qquad n \to \infty, 
\end{equation}
where $u_{0}=1$ and $u_k, k\geq 1$ can be determined recursively.
\end{corollary}

 \begin{corollary}\label{cor:full-asym-Dt/D0}
 The Hankel determinant $D_n\left(t_m,0,\cdots,t_s,\cdots ,t_1;\lambda    \right)$ has  has the following full asymptotic expansion
\begin{equation}
   \ln\frac{D_n(t_m,\cdots,t_{s+1},t_s,t_{s-1},\cdots ,t_1;\lambda)}{D_n(t_m,\cdots,t_{s+1},0,t_{s-1},\cdots ,t_1;\lambda)} \sim -2mt_{s}t_mN_{m+s}\sum_{k=0}^{\infty}v_{k,s}{n^{\frac{m+s-k}{m}}}, \qquad n \to \infty,
\end{equation}
where $v_{0,s}=1$ and $v_{k,s}, k\geq 1$ can be determined recursively.
\end{corollary}

\begin{example}
As an illustration, we choose $w\left(x; 1,0,0,t_{2},t_{1}; \lambda\right)$ as an example and state the asymptotic expansions of $\beta_{n},  \mathrm{p}\left(n;1,0,0,t_{2},t_{1}; \lambda\right)$ and 
$\ln D_{n}\left(1,0,0,t_{2},t_{1}; \lambda\right)$ as $n\to\infty$.

For the special weight $w\left(x;1,0,0,t_{2},t_{1};\lambda\right)=|x|^{2\lambda+1}\exp(-x^{10}-t_{2}x^4-t_1x^2)$ case, according to \eqref{eq:recurrence-equations}, the recurrence coefficient $\beta_{n}$ satisfies
\begin{equation}
\begin{split}
\beta_{n}=&\frac{n^{\frac{1}{5}}}{\kappa}-\frac{\kappa^2t_2}{525}n^{-\frac{2}{5}}-\frac{2t_1}{5\kappa^2}n^{-\frac{3}{5}}+\frac{\Delta_n}{5\kappa}n^{-\frac{4}{5}}-\frac{2\kappa t_1t_2}{2625}n^{-\frac{6}{5}}-\frac{\kappa^2(t_1^2-6\Delta_nt_2)}{7875}n^{-\frac{7}{5}}\notag\\
    	&+\left(\frac{48t_2^3}{4375\kappa^2}+\frac{6\Delta_nt_1}{25\kappa^2}\right)n^{-\frac{8}{5}}+\frac{1}{5\kappa}\left(\frac{1}{3}-\frac{2\Delta_n^2}{5}+\frac{16t_{1}t_{2}^2}{875}\right)n^{-\frac{9}{5}}\\
        &+\left(\frac{2t_1^3\kappa}{39375}+\frac{t_2^4\kappa}{47850}+\frac{4\Delta_nt_1t_2\kappa }{4375}\right)n^{-\frac{11}{5}}
        +\mathcal{O}(n^{-\frac{12}{5}})
    	\end{split}\label{example-decic-weight}
\end{equation}
as $n\to\infty$, where $\kappa=\sqrt[5]{1260}$ and $\Delta_{n}$ is defined by \eqref{eq:Delta-n}.

According to \eqref{24.20} again, the nontrivial leading coefficient $\mathrm{p}\left(n;1,0,0,t_1,t_2,\lambda\right)$ of the monic orthogonal polynomials satisfies
    \begin{align}
     	\mathrm{p}\left(n;1,0,0,t_1,t_2,\lambda\right)&=\frac{1}{2}[-2\lambda\beta_{n}\!-n\beta_{n}-4t_2\beta_{n-1}\beta_{n}\beta_{n+1}
     	-10(F_{6,n}-\beta_{n}F_{5,n})].\label{53.1}
     \end{align}
Substituting the asymptotic expansions of $\beta_{n}$ and $F_{k,n}$ into \eqref{53.1}, we have
     \begin{equation}
     \begin{split}
     	\mathrm{p}\left(n;1,0,0,t_1,t_2,\lambda\right)=&-\frac{5n^{\frac{6}{5}}}{6\kappa}+\frac{4t_2}{\kappa^3}n^{\frac{3}{5}}+\frac{t_1n^{\frac{2}{5}}}{\kappa^2}-\frac{\Delta_n+2\lambda}{2\kappa}n^{\frac{1}{5}}-\frac{2t_2^2}{175}-\frac{2\kappa t_1t_2}{525}n^{-\frac{1}{5}}\\
     	&-\left(\frac{t_1^2\kappa^2}{3150}-\frac{(2\lambda+\Delta_n)\kappa^2t_2}{1050} \right)n^{-\frac{2}{
5}}\\
&+\frac{175t_1(2\lambda+\Delta_n)+16t_2^3}{875\kappa^2}n^{-\frac{3}{5}}\\
&+\left(\frac{12t_1t_2^2+175}{2625\kappa}-\frac{(2\lambda+1)\Delta_n}{10\kappa}\right)n^{-\frac{4}{5}}\\
&+\frac{11\Delta_nt_2^2}{175}n^{-1}+\mathcal{O}(n^{-\frac{6}{5}})\label{53}\end{split}
     \end{equation}
as $n\to\infty$.

From \eqref{24.230010}, we have
    \begin{align}
    \frac{\partial}{\partial t_1}\ln D_n\left(1,0,0,0,t_1;\lambda\right)=[-10F_{6,n}-2t_1(\beta_{n})^2-(2\lambda+1-\Delta_n)\beta_n]|_{t_2=0}.\label{57}
    \end{align}
    Integrating both sides with respect to $t_1$, it follows that
    \begin{equation}
    \ln\frac{D_n\left(1,0,0,0,t_1;\lambda\right)}{D_n\left(1,0,0,0,0;\lambda\right)}=\int_{0}^{t_1}[-10F_{6,n}(s)-2t_1(\beta_{n}(s))^2-(2\lambda+1-\Delta_n)\beta_n(s)]|_{t_2=0}ds.\label{58.8}
    \end{equation}
    Substituting the asymptotic expansions of $\beta_n$ and $F_{6,n}$ into \eqref{58.8} and putting $t_2=0$, we obtain
    \begin{align}
    	\ln\frac{D_n\left(1,0,0,0,t_1;\lambda\right)}{D_n\left(1,0,0,0,0;\lambda\right)}=&-\frac{5t_1n^{\frac{6}{5}}}{3\kappa}+\frac{t_1^2n^{\frac{2}{5}}}{\kappa^2}-\frac{(2\lambda+1+\Delta_n)t_1n^{\frac{1}{5}}}{\kappa}- \frac{\kappa^2t_1^3}{4725} n^{
-\frac{2}{5}}\notag \\
&+\frac{(2\lambda+1+\Delta_n)t_1^2}{5\kappa^2}n^{-\frac{3}{5}}+ \left(\frac{2 t_1}{15\kappa}-\frac{(2\lambda+1)\Delta_nt_1}{5\kappa}\right) n^{-\frac{4}{5}}\notag\\
& +\mathcal{O}(n^{-1}).\label{60}
    \end{align}
    From \eqref{24.230010}, we have
    \begin{align}
    	&\frac{\partial}{\partial t_2}\ln D_n\left(1,0,0,t_2,t_1;\lambda\right)\notag\\
    	&=2t_1(\beta_{n-1}\beta_{n}\beta_{n+1}-\beta_{n}^2(\beta_{n-1}+\beta_{n}+\beta_{n+1}))-4t_2\beta_{n}^2(\beta_{n-1}+\beta_{n}+\beta_{n+1})^2\notag\\
    	&-10(\beta_{n}(R_{n-1,3}R_{n,3}+R_{n-1,4}R_{n,2}+R_{n-1,2}R_{n,4})-2r_{n,3}r_{n,4})\notag\\
        &-(2\lambda+1-\Delta_n)r_{n,2}.\label{60.12}
    \end{align}
Substituting the asymptotic expansion of $\beta_n$ in \eqref{example-decic-weight} into \eqref{60.12}
    and then integrating \eqref{60.12} with respect to $t_2$ over the interval $[0, t_2]$ , we find
    \begin{align}
    	\ln\frac{D_n\left(1,0,0,t_2,t_1;\lambda\right)}{D_n\left(1,0,0,0,t_1;\lambda\right)}=&-\frac{30t_2}{7\kappa^2}n^{\frac{7}{5}}+\frac{18t_2^2}{\kappa^4}n^{\frac{4}{5}}+\frac{8t_1t_2}{\kappa^3}n^{\frac{3}{5}}-\frac{3t_2(2\lambda+1+\Delta_n)}{\kappa^2}n^{\frac{2}{5}}-\frac{8t_2^3}{175\kappa}n^{\frac{1}{5}}\notag\\
        &+\frac{4t_1t_2^2}{175}-\frac{2\kappa t_1^2t_2-3t_2^2\kappa(2\lambda+1+\Delta_n)}{525}n^{-\frac{1}{5}}\notag\\
    &+\left(\frac{2\kappa^2t_2^4}{30625}+\frac{t_1t_2\kappa^2(2\lambda+1+\Delta_n)}{525}\right) n^{-\frac{2}{5}}\notag\\
    &+\left(\frac{3t_2(1+(4\lambda+2)\Delta_n^2)}{5\kappa^2}+\frac{32t_1t_2^3}{875\kappa^2}\right)n^{-\frac{3}{5}}+\mathcal{O}(n^{-\frac{4}{5}}),\label{62}
    \end{align}
A summation of \eqref{60} and \eqref{62} gives
        \begin{align}
    	\ln \frac{D_n\left(1,0,0,t_2,t_1;\lambda\right)}{D_n\left(1,0,0,0,0;\lambda\right)}=
        &-\frac{30t_2n^{\frac{7}{5}}}{7\kappa^2}-\frac{5t_1n^{\frac{6}{5}}}{3\kappa}+\frac{18t_2^2n^{\frac{4}{5}}}{\kappa^4}+\frac{8t_1t_2n^{\frac{3}{5}}}{\kappa^3}+\frac{t_1^2-3t_2(2\lambda+1+\Delta_n)}{\kappa^2}n^{\frac{2}{5}}\notag\\
        &-\frac{175(2\lambda+1+\Delta_n)t_1+8t_2^3}{175\kappa}n^{\frac{1}{5}}+\frac{4t_1t_2^2}{175}\notag\\
        &-\frac{2\kappa t_1^2t_2-3t_2^2\kappa(2\lambda+1+\Delta_n)}{525}n^{-\frac{1}{5}}\notag\\
        &+\left(\frac{2\kappa^2t_2^4}{30625}+\frac{t_1t_2\kappa^2(2\lambda+1+\Delta_n)}{525}- \frac{\kappa^2t_1^3}{4725}\right)n^{-\frac{2}{5}}\notag\\
        &+\left(\frac{3t_2(1+(4\lambda+2)\Delta_n^2)+(2\lambda+1+\Delta_n)t_1^2}{5\kappa^2}+\frac{32t_1t_2^3}{875\kappa^2}\right)n^{-\frac{3}{5}}\notag\\
        &+\mathcal{O}(n^{-\frac{4}{5}}),\label{56}
    \end{align}
    where $\kappa=\sqrt[5]{1260}$.
\end{example}

In Corollary \ref{cor:full-asym-Dt/D0}, we only state the full asymptotic expansion of $\ln\frac{D_{n}(t_{m},t_{m-1}\cdots,t_{1};\lambda)}{D_{n}(t_{m},0,\cdots,0;\lambda)}$ as $n\to\infty$. In order to get the full asymptotic expansion of $\ln D_n\left(t_{m},t_{m-1},\cdots,t_{1};\lambda\right)$ as $n\to\infty$, we still need the one of $D_{n}(t_{m},0,\cdots,0;\lambda)$. Nevertheless, to the best of our knowledge, in literature, we only have the large $n$ asymptotic behaivor of $D_{n}(t_{m},0,\cdots,0;\lambda)$ with leading terms. Precisely, from \cite{Charlier-Gharakhloo,22}, we have
\begin{align}
     \ln D_{n}(t_{m},0,\cdots,0;\lambda) =&\frac{n^2\ln{n}}{2m} +\left(\ln{\frac{A}{2}}-\frac{3}{4m}-\frac{\ln{t_m}}{2m}\right)n^2+\frac{2\lambda+1}{2m}n\ln{n}\notag\\
     &+\left(\ln{(2\pi)}+(2\lambda+1)\ln{\frac{A}{2}}-\frac{(2\lambda+1)t_m}{2m}\right)n+\left(\frac{(2\lambda+1)^2}{4}-\frac{1}{12}\right)\ln{n}\notag\\
     &+\zeta^{'}(-1)-\frac{\ln{m}}{12}+\frac{(2\lambda+1)^2}{4}\ln{\left(\frac{2mA}{2m-1}\right)}+\ln{\left(\frac{G^2(\lambda+\frac{3}{2})}{G(2\lambda+2)}\right)}+\mathcal{O}\left(\frac{\ln{n}}{n}\right) \label{eq:leading-asym-Dn-reduced}
   \end{align}
   as $n\to\infty$, where $A=\left(t_{m}/A_{m}\right)^{\frac{1}{m}}$ and $A_{m}$ is defined in \eqref{eq:value-Ak}. To achieve the full asymptotic expansion of $D_{n}(t_{m},0,\cdots,0;\lambda)$ as $n\to\infty$, we shall use the following formula 
\begin{equation}\label{eq:relation-betan-Dn}
    \ln\beta_{n}=\ln D_{n+1}(T_{m};\lambda)+\ln D_{n-1}(T_{m};\lambda)-2\ln D_{n}(T_{m};\lambda)
\end{equation}
which is derived form the combination of \eqref{eq:relation-Dn-hn} and \eqref{eq:relation-betan-hn}. The corresponding analysis consists of two steps. 

First, we prove the following theorem.
\begin{theorem}\label{thm:beta-n-full-reduced}
    Let $t_j = 0$ for $j = 1, 2, \dots, m-1$, let $t_m > 0$, and let $\lambda>-1$. Then, for any $p \in \mathbb{N}^+$, the recurrence coefficient $\beta_n$ satisfies
    \begin{equation}\label{eq:ln-beta-n-reduced}
        \ln \beta_n = \frac{1}{m} \ln n + \ln N_1 + \sum_{i=1}^{p} \frac{b_i}{n^{i}} + \mathcal{O}\!\left(\frac{1}{n^{p+1}}\right)
    \end{equation}
    as $n \to \infty$.
\end{theorem}

\begin{proof}
    To establish \eqref{eq:ln-beta-n-reduced}, it suffices to prove the full asymptotic expansion
    \begin{equation}
        \label{eq:beta-n-reduced-full}
        \beta_n \sim N_1 n^{\frac{1}{m}} \sum_{i=0}^{\infty} \frac{\tilde{b}_i}{n^{i}}, \qquad n \to \infty,
    \end{equation}
    with $\tilde{b}_0 = 1$. Substituting \eqref{eq:beta-n-reduced-full} into \eqref{eq:definition-dPI} and set $t_{1}=t_{2}=\cdots=t_{m-1}=0$, we find that the asymptotic expansion \eqref{eq:beta-n-reduced-full} can be formally derived, and all coefficients $\tilde{b}_i$ for $i \in \mathbb{N}^+$ can be determined recursively.

    Define
    \begin{equation}
        \label{eq:beta-n-reduced-E_{p}}
        \beta_n = N_1 n^{\frac{1}{m}} \left( 1 + \sum_{i=1}^{p} \frac{\tilde{b}_i}{n^{i}} + \frac{E_p(n)}{n^{p+1}} \right).
    \end{equation}
    It remains to prove that, for any $p\in\mathbb{N}^{+}$
    \begin{equation}\label{eq:Epn-reduced}
        E_p(n) = \mathcal{O}(1), \qquad n \to \infty,
    \end{equation}
    which we shall do by induction on $p$.

    For $p = 1$, the estimate \eqref{eq:Epn-reduced} can be verified using the Coulomb fluid method described in Section~\ref{sec: Large-n-asym}. Now assume that \eqref{eq:Epn-reduced} holds for $p = 1, \dots, M$; we will show that it also holds for $p = M+1$. To this end, we first prove that $E_{M+2m-2}(n) = O(n^{2m-2})$ as $n \to \infty$. Indeed, by the induction hypothesis we have $E_{M+2m-2}(n) = O(n^{2m-1})$.

    Substituting \eqref{eq:beta-n-reduced-E_{p}} with $p = M+2m-2$ into \eqref{eq:definition-dPI}, we obtain a linear difference equation of the form
    \begin{equation}
        \sum\limits_{k=-m+1}^{m-1}\alpha_{k} E_{M+2m-2}(n+k)= g_{M+2m-2}(n),
    \end{equation}
    where $g_{M+2m-2}(n)$ is a polynomial of the coefficients $\tilde{b}_i$ for $i = 1, \dots, M+2m-2$ and the terms $E_{M+2m-2}(n+j)$ for $j = -m+1, \dots, m-1$, and satisfies $g_{M+2m-2}(n) = O(n^{2m-2})$ as $n \to \infty$. By the standard theory of linear difference equations with constant coefficients \cite{Elaydi-difference}, it follows that $E_{M+2m-2}(n) = O(n^{2m-2})$ as $n \to \infty$. Substituting this improved estimate into \eqref{eq:beta-n-reduced-E_{p}} yields \eqref{eq:Epn-reduced} for $p = M+1$. 
    
    This complete the proof of Theorem \ref{thm:beta-n-full-reduced}.
\end{proof}

Next, we show the full asymptotic expansion of $\ln D_{n}(T_{m};\lambda)$ as $n\to\infty$ with $t_{j}=0, j=1,2,\cdots,m-1$.

\begin{theorem}\label{thm:full-asym-Dn-reduced}
  Let $t_{j}=0, j=1,2\cdots,m-1$, $t_{m}>0$ and $\lambda>-1$. Then, for any $p\in\mathbb{N}^{+}$, we have
   \begin{align}
     \ln D_{n}(T_{m};\lambda)=&\tilde{d}_{-2}n^2\ln{n} +d_{-2}n^2+\tilde{d}_{-1}n\ln{n}+d_{-1}n+\tilde{d}_{0}\ln{n}+\sum\limits_{i=0}^{p-1}\frac{d_{i}}{n^{i}}+\mathcal{O}\left(\frac{1}{n^{p}}\right)\label{eq:full-expansion-Dn-reduced}  
   \end{align}
   as $n\to\infty$, where
   \begin{equation}
   \begin{split}
       \tilde{d}_{-2}&=\frac{1}{2m},\qquad d_{-2}=-\frac{3}{4m}-\ln{2A_{m}}, \quad \tilde{d}_{-1}=\frac{2\lambda+1}{2m},\\
       d_{-1}&=\ln{(2\pi)}+\frac{2\lambda+1}{2m}\ln{t_{m}}-\frac{2\lambda+1}{2m}\ln{2A_{m}}-\frac{(2\lambda+1)t_m}{2m},\quad \tilde{d}_{0}=\left(\frac{(2\lambda+1)^2}{4}-\frac{1}{12}\right)\\
       d_{0}&=\zeta^{'}(-1)-\frac{\ln{m}}{12}+\frac{(2\lambda+1)^2}{4}\ln{\left(\frac{2m}{2m-1}\right)}+\frac{(2\lambda+1)^2}{4m}\ln{\frac{t_{m}}{A_{m}}}+\ln{\left(\frac{G^2(\lambda+\frac{3}{2})}{G(2\lambda+2)}\right)},\\
       d_{1}&=\frac{b_{3}}{2}-\frac{\tilde{d}_{-1}}{12},\qquad d_{2}=\frac{b_{4}}{6}+\frac{\tilde{d}_{0}}{12}+\frac{\tilde{d}_{-2}}{180},
       \end{split}
   \end{equation}
   $A_{m}$ is defined in \eqref{eq:value-Ak}, $b_{i}$'s are given in \ref{thm:beta-n-full-reduced} and $G(\cdot)$ is the Barnes-$G$ function. 
   
\end{theorem}

Before the proof of this theorem, we first give the following lemma.
\begin{lemma}\label{lem:approx-Gkn}
    For any $k\in\mathbb{N}^{+}$, let $G_{k}(n)$ be a specific solution of the following difference equation 
    \begin{equation}
        \frac{G_{k}(n+1)}{(n+1)^{k}}+\frac{G_{k}(n-1)}{(n-1)^{k}}-2\frac{G_{k}(n)}{n^{k}}=\frac{c(n)}{n^{k+2}},
    \end{equation}
    where $c(n)=\mathcal{O}(1)$ as $n\to\infty$. If $G_{k}(n)=o(n^{k})$ as $n\to\infty$, then $G_{k}(n)=\mathcal{O}(1)$ as $n\to\infty$. 
\end{lemma}

\begin{proof}
    Set 
    \begin{equation}
        H_k(n) = \frac{G_k(n+1)}{(n+1)^k} - \frac{G_k(n)}{n^k}.
    \end{equation}
    Then the difference equation becomes 
    \begin{equation}\label{eq:difference-Hkn}
        H_k(n) - H_k(n-1) = \frac{c(n)}{n^{k+2}}.
    \end{equation}
    Since $G_k(n) = o(n^k)$ as $n \to \infty$, we have 
    $\frac{G_k(n)}{n^k} \to 0$, which implies $\lim\limits_{n \to \infty} H_k(n) = 0$. 
    
    Summing \eqref{eq:difference-Hkn} from $m=n$ to $N-1$ yields
    \begin{equation}
        H_k(N) - H_k(n) = \sum_{m=n}^{N-1} \frac{c(m)}{m^{k+2}}.
    \end{equation}
    Letting $N \to \infty$ and making use of the fact that $\lim\limits_{N \to \infty} H_k(N) = 0$, we obtain
    \begin{equation}
        H_k(n) = -\sum_{m=n}^{\infty} \frac{c(m)}{m^{k+2}}.
    \end{equation}
    Since $c(n) = \mathcal{O}(1)$, there exists $C > 0$ such that for sufficiently large $n$,
    \begin{equation}
        |H_k(n)| \leq \sum_{m=n}^{\infty} \frac{C}{m^{k+2}} \leq \frac{C}{(k+1)n^{k+1}}.
    \end{equation}
    
    Now note that $H_k(m) = \frac{G_k(m+1)}{(m+1)^k} - \frac{G_k(m)}{m^k}$, so by the similar argument,
    \begin{equation}
        \frac{G_k(N)}{N^k} - \frac{G_k(n)}{n^k} = \sum_{m=n}^{N-1} H_k(m).
    \end{equation}
    Letting $N \to \infty$ and using $\lim\limits_{N \to \infty} \frac{G_k(N)}{N^k} = 0$, we get
    \begin{equation}
        \frac{G_k(n)}{n^k} = -\sum_{m=n}^{\infty} H_k(m).
    \end{equation}
    Therefore, there exist constants $C', C''$ such that for sufficiently large $n$,
    \begin{equation}
        \left|\frac{G_k(n)}{n^k}\right| \leq \sum_{m=n}^{\infty} |H_k(m)| \leq \sum_{m=n}^{\infty} \frac{C'}{m^{k+1}} \leq \frac{C''}{n^k},
    \end{equation}
    which implies $|G_k(n)| \leq C''$. 
    This finishes the proof of Lemma \ref{lem:approx-Gkn}.
\end{proof}

\textbf{Proof of Theorem \ref{thm:full-asym-Dn-reduced}.}
The leading terms in \eqref{eq:full-expansion-Dn-reduced} can be derived from literature \cite{Charlier-Gharakhloo,22};  Now we intend to show that all the coefficients $d_{i}, i\in\mathbb{N}$ can be derived formally and the asymptotic approximation in \eqref{eq:full-expansion-Dn-reduced} is ture for any $p\in\mathbb{N}$.

Substituting \eqref{eq:full-expansion-Dn-reduced} and \eqref{eq:beta-n-reduced-full} into \eqref{eq:relation-betan-Dn}, we have 
\begin{equation}
    \begin{split}
      \frac{1}{6}\tilde{d}_{-1}+2d_{1}=b_{3},\qquad -\frac{1}{30}\tilde{d}_{-2}-\frac{1}{2}d_{0}+6d_{2}=b_{4},
    \end{split}
\end{equation}
and for $j>2$
\begin{equation}\label{eq:di-recurrence}
 \tilde{f}_{j}(\tilde{d}_{-2},d_{-2},\tilde{d}_{-1},d_{-1},d_{0},\cdots,d_{j-2})+j(j+1)d_{j}=b_{j+2}.   
\end{equation}
Here $\tilde{f}_{j}$ is a function of $\tilde{d}_{-2},d_{-2},\tilde{d}_{-1},d_{-1},d_{0},\cdots,d_{j-1}$ which is much complicate and hence we do not stated its general representation. Nevertheless, from \eqref{eq:di-recurrence}, we conclude that all the coefficients $d_{j}$ can be calculated recursively.

Denote 
\begin{equation}\label{eq:def-Gp(n)}
   \ln D_{n}(T_{m};\lambda)=\tilde{d}_{-2}n^2\ln{n} +d_{-2}n^2+\tilde{d}_{-1}n\ln{n}+d_{-1}n+\tilde{d}_{0}\ln{n}+\sum\limits_{i=0}^{p-1}\frac{d_{i}}{n^{i}}+\frac{G_{p}(n)}{n^{p}}. 
\end{equation} 
Now we show that $G_{p}(n)=\mathcal{O}(1)$ as $n\to\infty$ by induction of $p$. 

When $p=1$, from \eqref{eq:leading-asym-Dn-reduced}, we see that $G_{1}(n)=\mathcal{O}(\log{n})$ as $n\to\infty$. Substituting \eqref{eq:def-Gp(n)} into \eqref{eq:relation-betan-Dn}, we have 
\begin{equation}
   \frac{G_{1}(n+1)}{n+1}+\frac{G_{1}(n-1)}{n-1}- 2\frac{G_{1}(n)}{n}=\frac{c_{1}(n)}{n^3}=\frac{b_{3}}{n^{3}}-\frac{\tilde{d}_{-1}}{6n^{3}}+\mathcal{O}(n^{-4})
\end{equation}
as $n\to\infty$. Hence, from Lemma \ref{lem:approx-Gkn}, we have $G_{1}(n)=\mathcal{O}(1)$ as $n\to\infty$. 

Assume that $G_{m}(n)=\mathcal{O}(1)$ as $n\to\infty$ for all $m=1,2,\cdots,p$. Now we intend to show that $G_{p+1}(n)=\mathcal{O}(1)$ as $n\to\infty$.

By the inductive assumption, we see that $G_{p+1}(n)=\mathcal{O}(n)$ as $n\to\infty$. Replace $p$ by $p+1$ in \eqref{eq:def-Gp(n)} and substitute it into \eqref{eq:relation-betan-Dn}, we find that
\begin{equation}
   \frac{G_{p+1}(n+1)}{(n+1)^{p+1}}+\frac{G_{p+1}(n-1)}{(n-1)^{p+1}}-2\frac{G_{p+1}(n)}{n^{p+1}}=\frac{c_{p}(n)}{n^{p+3}}
\end{equation}
as $n\to\infty$, where $c_{p}(n)=\mathcal{O}(1)$ as  $n\to\infty$. Applying Lemma \ref{lem:approx-Gkn} again, we obtain the desired result $G_{p+1}(n)=\mathcal{O}(1)$ as $n\to\infty$.

\begin{remark}
    From Theorem \ref{thm:full-asym-Dn-reduced}, we find that although $\ln{n}$ appears in the leading terms of the asymptotic expansion of $\ln D_{n}(T_{m};\lambda)$ when $t_{m}>0$, $\lambda>-1$, and $t_{j}=0$ for $j=1,2,\dots,m-1$, it is absent from all higher-order terms. Combining this observation with Corollary \ref{cor:full-asym-Dt/D0}, we see that the same phenomenon occurs in the asymptotic expansion of the Hankel determinant for the general case $t_{m}>0$, $\lambda>-1$, and $t_{j}\neq 0$ ($j=1,2,\dots,m-1$).
\end{remark}

\section*{Competing interests}
No competing interest is declared.

\section*{Author contributions statement}
All authors have contributed equally to this work. All authors have read and agreed to the published version of the manuscript.

\section*{Acknowledgements}
Wen-Gao Long was partially supported by the National Natural Science Foundation of China [Grant No. 12401094], the Natural Science Foundation of Hunan Province [Grant No. 2024JJ5131] and the Outstanding Youth Fund of Hunan Provincial Department of Education [Grant No. 23B0454]. Chao Min was partially supported by the National Natural Science Foundation of China [Grant No. 12001212] and the Fundamental Research Funds for the Central Universities [Grant No. 26JCPY009].

\end{document}